\documentclass[reqno, 11pt]{article}

\usepackage{fullpage}

\usepackage[centertags]{amsmath}
\usepackage {amsmath}
\usepackage{amsfonts}
\usepackage{amssymb}
\usepackage{amsthm}
\usepackage{enumerate}
\usepackage{mdframed}
\usepackage{comment}

\usepackage{newlfont}
\usepackage{color}

\usepackage{authblk}

\usepackage{stmaryrd}
\SetSymbolFont{stmry}{bold}{U}{stmry}{m}{n}

\usepackage{bbm}
\include{amsthm_sc}

\usepackage{stmaryrd}
\usepackage{mathrsfs}

\usepackage{fancyhdr}
\usepackage{graphicx}
\usepackage{fancybox}
\usepackage{setspace}
\usepackage{cleveref}

\newtheorem{thm}{Theorem}

\newtheorem{lem}[thm]{Lemma}
\newtheorem{prop}[thm]{Proposition}

\theoremstyle{definition}
\newtheorem{defn}[thm]{Definition}

\theoremstyle{remark}
\newtheorem{rem}[thm]{Remark}

\newcommand{\R} {\mathbb{R}}

\newcommand{\E} {\mathbb{E}}

\newcommand{\p} {\mathbb{P}}

\DeclareMathOperator{\Tr}{Tr}

\newcommand{\caE}{{\mathcal E}}

\newcommand{\caZ}{{\mathcal Z}}

\newcommand{\bss}{{\boldsymbol s}}

\newcommand{\wt}{\widetilde}
\newcommand{\wh}{\widehat}

\newcommand{\beq}{ \begin{equation} }
\newcommand{\eeq}{ \end{equation} }

\newcommand{\dd}{\mathrm{d}}
\newcommand{\ii}{\mathrm{i}}

\renewcommand{\bss}{\boldsymbol{\sigma}}

\numberwithin{equation}{section}
\numberwithin{thm}{section}

\title{Fluctuations of the free energy of the spherical Sherrington--Kirkpatrick model with sparse interaction}
\author{Haram Kim\footnote{email: \texttt{rlagkfka1221@kaist.ac.kr}} and Ji Oon Lee\footnote{email: \texttt{jioon.lee@kaist.edu}}}
\affil{\textit{Department of mathematics, KAIST}}
\date{\today}

\begin{document}
\maketitle

\begin{abstract}
We consider the spherical Sherrington--Kirkpatrick model of spin glass with sparse interaction, where the interactions between most of the pairs of the spin variables are possibly zero. With suitable normalization, we prove that the limiting free energy does not depend on the sparsity whereas the fluctuation of the free energy does. We also prove that both in the high- and the low-temperature regimes the fluctuation of the free energy converges in distribution to Gaussian distributions of same order when the sparsity is on a certain level, but their variances are different.
\end{abstract}

\section{Introduction}

The spin glass model was first introduced by Edwards and Anderson \cite{edwards1975theory} as a mathematical model to describe dilute magnetic alloys. The mean-field analogue of the model of Edwards and Anderson was considered by Sherrington and Kirkpatrick \cite{sherrington1975solvable}, which is defined by the Hamiltonian
\beq \label{eq:Hamiltonian}
	H_N(\bss) = \sum_{i, j=1}^N M_{ij} \sigma_i \sigma_j,
\eeq
where the spin variable $\bss = (\sigma_1, \dots, \sigma_N) \in \{ \pm 1 \}^N$ and $M$ is an $N \times N$ real symmetric random matrix. The Sherrington--Kirkpatrick (SK) model has been intensively studied in statistical physics and probability theory to analyze not only the behavior of spin glass but also other related models in high-dimensional statistics and learning theory. We refer to \cite{mezard2009information} and references therein for more information on the application of spin glass in various fields of study.

One of the main differences between the SK model and the Edwards--Anderson (EA) model is the density (or the sparsity) of the interaction; we can say that the interaction in the SK model is denser that that of the EA model, since most of the interactions $M_{ij}$ are non-zero, whereas $M_{ij}$ does not vanish only if the sites $i$ and $j$ are adjacent in the EA model. While the dense nature of the SK model make the analysis relatively easier, it may not capture certain properties of the spin glass due to the density of the interaction.

In this paper, we study how the sparsity affects the spin glass model by considering sparse random matrices. One of the most relevant examples of sparse random matrices is a diluted Wigner matrix, which is defined as follows: for $i \leq j$, we let $M_{ij} = B_{ij} W_{ij}$, where $W_{ij}$ are i.i.d. centered random variables with unit variance whose all moments are finite, and $B_{ij}$ are i.i.d. Bernoulli-type random variables, independent of $W$, satisfying
\[
	\p \left( B_{ij} = (Np)^{-1/2} \right) = p, \quad \p \left( B_{ij} = 0 \right) = 1-p.
\]
In this model, if $p \ll 1$, most of $M_{ij}$ vanish, and hence the interaction is sparse. For some technical reasons, we define the sparsity of the diluted Wigner matrix by $q:=\sqrt{Np}$. Note that $M$ reduces to a Wigner matrix if $p=1$ (or $q=\sqrt{N}$). See Definition \ref{def:Sparse Matrix} for the precise definition of the sparse random matrix and its sparsity.

For qualitative analysis, we focus on the spherical variant of the SK model, known as the spherical Sherrington--Kirkpatrick (SSK) model. First introduced by Kosterlitz, Thouless, and Jones \cite{kosterlitz1976spherical}, in the SSK model the Hamiltonian is given by \eqref{eq:Hamiltonian} as in the usual SK model, but the spin variable is on the $N$-dimensional Hypersphere with radius $\sqrt{N}$, defined by $S_{N-1} := \{ \bss \in \R^N : \| \bss \|^2 = N \}$, instead of the hypercube. (See Definition \ref{def:partition} for more detail.) The analysis of the SSK model is directly related to that of the eigenvalues of the random matrix $M$. The SSK model is also of great interest, partially due to its canonical connection with the likelihood ratio (LR) test in high-dimensional statistics. We refer to \cite{onatski2013asymptotic,el2020fundamental,chung2019weak} and references therein for more discussion on it.

With the aid of recent developments of random matrix theory, for several fundamental cases including when $M$ is a Wigner matrix, it is possible to prove various properties of the model, most notably the phase transition of the free energy, which is defined by
\beq
	F_N \equiv F_N(\beta) = \frac{1}{N} \int e^{\beta H_N(\bss)} \dd \omega_N(\bss),
\eeq
where $\beta$ is the inverse temperature and $\dd \omega_N$ is the normalized uniform measure on $S_{N-1}$. If $M_{ij}\ (i \leq j)$ are i.i.d. centered Gaussian random variables with variance $N^{-1}$, the it was proved in \cite{baik2016fluctuations} that
\beq \label{eq:limiting_free_energy}
	F_N(\beta) \to
		\begin{cases}
	\beta^2 & \text{ if } 0<\beta \leq 1/2, \\
	2\beta - \frac{\log (2\beta) + \frac{3}{2}}{2} & \text{ if } \beta > 1/2,
	\end{cases}
\eeq
as $N \to \infty$. It is also known that fluctuation of the free energy converges in distribution to a Gaussian of order $N^{-1}$ in the high temperature regime $(\beta < 1/2)$ and a GOE Tracy--Widom law of order $N^{-2/3}$ in the low temperature regime $(\beta > 1/2)$.

Our main goal in this work is to find how the limit and the fluctuation of the free energy change as the sparsity of $M$ increases.

\subsection{Main contribution}

In this paper, assuming the sparsity $q=N^{\phi}$ for some $\phi \in (0, 1/2)$, we prove the following: 

\begin{itemize}
\item (Theorem \ref{thm:third}) the limiting free energy converges to the same limit as in \eqref{eq:limiting_free_energy},
\item (Theorem \ref{thm:sub}) in the high temperature regime $(\beta < 1/2)$, the fluctuation of the free energy converges in distribution to a Gaussian distribution of order $N^{-\phi - \frac{1}{2}}$, and
\item (Theorem \ref{thm:sup}) in the low temperature regime $(\beta > 1/2)$, the fluctuation of the free energy converges in distribution to the sum (convolution) of a Gaussian distribution of order $N^{-\phi - \frac{1}{2}}$ and the GOE Tracy--Widom distribution of order $N^{-2/3}$.
\end{itemize}
We also provide precise formulas for the limiting distributions. We remark that while the fluctuation of the free energy converges to Gaussian distributions of same order both in the high and the low temperature regimes when $q \ll N^{1/6}$, their variances are different.

For the proof of the main results, we follow the strategy introduced in \cite{baik2016fluctuations}. We first write the partition function as a contour integral of an exponential function whose randomness depends on $M$ only through its eigenvalues $\lambda_1 \geq \lambda_2 \geq \dots \geq \lambda_N$. (See Proposition \ref{prop:integral}.) We then apply the method of the steepest-descent for the contour integral. In the high temperature case, it can be shown that the free energy can be approximated by a linear spectral statistics (LSS) of $M$ with respect to a logarithmic function. On the other hand, again by applying the steepest-descent method to the low temperature case, we show that the free energy is governed by the largest eigenvalue $\lambda_1$. With the results from random matrix theory on the LSS (Proposition \ref{cond:Linear Statistics}) and on the limit of the largest eigenvalue (Proposition \ref{cond:Limit of the Largest Eigenvalue}), we can prove our main results.

It is universal the dichotomy between the fluctuation given by the LSS in the high temperature regime and that by the largest eigenvalue in the low temperature regime as it holds with most of the classical random matrix models \cite{baik2016fluctuations}. However, it should be noted that the detail of the fluctuations is not universal and may depend on the random matrix model. Indeed, if $q \ll N^{1/6}$, the limiting fluctuation in the low temperature regime is Gaussian, instead of the GOE Tracy--Widom law as in the Wigner case. Such a non-universal behavior of the free energy was first observed in the deformed model where $M$ is a deformed Wigner matrix of the form $W + \lambda V$ with a Wigner matrix $W$ and a random diagonal matrix $V$ with i.i.d. entries. In this case, the limiting fluctuation of the free energy in the low temperature regime can be Gaussian or Weibull, depending on $\lambda$ and the decay of the diagonal entries of $V$. (See \cite{lee2023spherical} for more detail.)

The most interesting feature in the SSK model with sparse interaction is the order of the limiting fluctuations; for our best knowledge, it provides the first instance of spin glass models where the orders of the fluctuations in both the high and the low temperature regimes coincide (when $q \ll N^{1/6}$). 
Heuristically, this is mainly due to the fact that the spectral behavior of $M$ is largely governed by the random variable
\beq \label{eq:caZ}
	\caZ:=\frac{1}{N} \Tr M^2-1.
\eeq
From our definition of the sparse random matrix in Definition \ref{def:Sparse Matrix}, it can be readily checked that $\caZ$ is typically of order $N^{-\phi-\frac{1}{2}}$. Heuristically, the fluctuation originated from $\caZ$ is shared by all eigenvalues. Even for the largest eigenvalue, if $q \ll N^{1/6}$, the fluctuation due to $\caZ$ dominates the Tracy--Widom fluctuation of the largest eigenvalue, which is of order $N^{-2/3}$.

The main technical difficulty in the proof of our main results is the fact that the semicircle measure $\nu_{sc}$ does not approximate the empirical spectral measure $N^{-1} \sum_{j=1}^N \delta_{\lambda_j}$ well. While the empirical spectral measure converges in distribution to the semicircle measure, for more detailed analysis involving the eigenvalues of sparse matrices, it is required to use another measure $\nu$, which can be considered deterministic refinement of the semicircle measure. (See Proposition \ref{cond:regular} for more detail.) While the difference between $\nu$ and $\nu_{sc}$ is small, it is not negligible for our purpose. (See Proposition \ref{measure difference} for more detail.)

Another technical challenge is due to the rigidity, which provides estimates on the differences between the deterministic locations predicted by the limiting empirical distribution and the actual random location of the eigenvalues. Such estimates were the crucial input to control the errors in the analysis of the steepest-descent in \cite{baik2016fluctuations}. For a sparse matrix, however, the estimate based on the deterministic locations deteriorates as the sparsity parameter $q$ decreases, and it is not directly applicable in the analysis when $q$ is below a certain level. It turns out that the rigidity estimate can be improved by further considering the randomness originated from $\caZ$. (See Proposition \ref{cond:rigidity} for more detail.) While this improved rigidity estimate also deteriorates as $q$ decreases, it is possible to control the error terms appearing in the proof.

Since the fluctuation of the LSS and that of $\caZ$ are of the same order, it is natural to ask whether the former is solely governed by the latter. In this paper, we show that the answer is affirmative by applying the rigidity estimate involving $\caZ$ to prove the asymptotic normality of the LSS. See Proposition \ref{cond:Linear Statistics} for the precise statement on the convergence of the LSS and Section \ref{sec:LSS} for its proof.

\subsection{Related works}

The SK model was first introduced by Sherrington and Kirkpatrick \cite{sherrington1975solvable}. The limiting free energy for the SK model was first predicted by Parisi \cite{parisi1979infinite,parisi1980order} with replica symmetry breaking. The formula for the limiting free energy is known as the Parisi formula and rigorously proved by Guerra \cite{guerra2003broken} and Talagrand \cite{talagrand2006parisi}. The fluctuation of the free energy of the SK model in the high temperature regime was proved by Aizenman, Lebowitz, and Ruelle \cite{aizenman1987some}, which coincides with that of the SSK model in the high temperature regime. 

The SSK model was first introduced by Kosterlitz, Thouless, and Jones \cite{kosterlitz1976spherical}. The limiting free energy for the SSK model was obtained in \cite{kosterlitz1976spherical} (without a rigorous proof) by applying random matrix theory. The SSK model was generalized to the $p$-spin model by Crisanti and Sommers \cite{crisanti1992spherical}, where a formula analogous to the Parisi formula was also obtained. The limiting free energy was proved by Talagrand \cite{talagrand2006free}. The fluctuation of the free energy for the SSK model was proved in \cite{baik2016fluctuations}, and later extended to other related spherical spin glass models including the Curie--Weiss model \cite{baik2017fluctuations} and the bipartite SSK model \cite{baik2020free}. The near-critical behavior of the free energy \cite{landon2022free,johnstone2021spin} and the overlap \cite{landon2022fluctuations} are also proved for the SSK model. For more results on the spherical spin glass models, we refer to \cite{baik2021spherical} and the references therein.

The sparse random matrices has been extensively studied as a generalization of the adjacency matrices of the sparse Erd\H{o}s-R\'enyi graphs, which are widely applied in various fields including combinatorial optimization \cite{Mohar1993EigenvaluesIC}, spectral partitioning \cite{janson2011random}, and community detection \cite{Applicationmethodology}. The study of the sparse random matrices in random matrix theory was initiated in a series of papers by Erd\H{o}s, Knowles, Yau, and Yin \cite{erdos2013spectral,erdHos2012spectral}. The central limit theorem for the LSS of the sparse Erd\H{o}s-R\'enyi graphs was proved by Shcherbina and Tirozzi \cite{shcherbina2012central}. The Tracy--Widom limit of the largest eigenvalues of sparse random matrices was proved for the case $q \gg N^{1/3}$ in \cite{erdHos2012spectral} and extended to the case $q \gg N^{1/6}$ in \cite{lee2018local}. The dominance of the random variable $\caZ$ in \eqref{eq:caZ} in the spectral fluctuation was first observed by Huang, Landon, and Yau in \cite{huang2020transition}, where it was also proved the Gaussian convergence of the largest eigenvalue for the case $N^{1/9} \ll q \ll N^{1/6}$. The result was further extended in \cite{he2021fluctuations,lee2021higher,huang2022edge}.

\subsection{Organization of the paper}

The rest of the paper is organized as follows: In Section \ref{sec:main}, we precisely define the mathematical model and state our main results followed by the main idea of the proof. In Section \ref{sec:prelim}, we collect previously known results about sparse random matrices that are suitable in our analysis. In Sections \ref{sec:high} and \ref{sec:low}, we prove our main results in the high temperature regime and the low temperature regime, respectively. In Section \ref{sec:LSS}, we prove the LSS by using the random variable $\mathcal{Z}$. The proof of Proposition \ref{measure difference}, and some technical details in Sections \ref{sec:high} and \ref{sec:low} are provided in Appendix.

\section{Definitions and Main Results} \label{sec:main}

In this section, we define the model and state our main results.
\subsection{Definitions}
We begin by precisely defining the sparse matrix. 
\begin{defn}[Sparse Matrix]\label{def:Sparse Matrix}
We say an $N\times N$ real symmetric matrix $M$ is a sparse matrix with sparsity parameter $q:=N^{\phi}$ for a constant $0 < \phi < 1/2$ if its upper triangle entries $M_{ij}$ ($i\le j$) are independent real random variables satisfying the following conditions:
\begin{itemize}
\item For any $i, j$, $\E [ M_{ij}] = 0$.
\item For any $i < j$, $\E [ |M_{ij}|^2] = \frac{1}{N}$.
\item For any $i, j$ and any integer $k \ge 2$, there exist ($N$-independent) constants $C_1, C_2 > 0$ such that $C_1 N^{-(1+(k-2)\phi)} \le \E [ |M_{ij}|^k ] \le C_2 N^{-(1+(k-2)\phi)}$.
\end{itemize}
\end{defn}
Note that in our definition of sparse matrices the off-diagonal entries of $M$ are not necessarily identically distributed. We also notice that if we set the sparsity parameter $q = \sqrt{N}$ in Definition \ref{def:Sparse Matrix} then the random matrix $M$ becomes a Wigner matrix.

As discussed in Introduction, the spectral behavior of $M$ is largely governed by the random variable
\[
	\caZ:=\frac{1}{N} \Tr M^2-1.
\]
It is then immediate to check that $\caZ/(\sqrt{2}\Sigma)$ converges in distribution to a standard Gaussian as $N \to \infty$, where we define
\beq \label{eq:Sigma}
	\Sigma:=\left(\frac{1}{N^2}\sum_{i,j=1}^N \E [M_{ij}^4] \right)^{1/2}.
\eeq
For the sake of simplicity, we assume that 
\beq \label{eq:sigma}
	\lim_{N\to \infty} N^{\phi+\frac{1}{2}}\Sigma = \sigma.
\eeq

We next define the partition function and the free energy of the SSK model spin system.
\begin{defn}[Free energy]\label{def:partition}
For an $N\times N$ real symmetric matrix $M=(M_{ij})_{i,j=1}^N$, we define the partition function at inverse temperature $\beta>0$ by 
\beq \label{eq:partition function}
	Z_N \equiv Z_N(\beta) := \int_{S_{N-1}} e^{\beta \langle \bss, M\bss \rangle } \dd \omega_N(\bss), \qquad \langle \bss, M\bss \rangle = \sum_{i,j=1}^N M_{ij} \sigma_i\sigma_j.
\eeq
Here, we denote by $\dd\omega_N$ the Haar (normalized uniform) measure on the sphere $S_{N-1} := \{ \bss \in \R^N : \| \bss \|^2 = N \}$.
The free energy $F_N$ is defined by
\beq \label{eq:freeenergydef}
	F_N \equiv F_N(\beta) := \frac{1}{N} \log Z_N.
\eeq 
\end{defn}

\subsection{Main Results}

Our first main result is about the limiting free energy.

\begin{thm}[Limiting Free energy] \label{thm:third}
Suppose that $M$ is a sparse matrix defined in Definition \ref{def:Sparse Matrix}. Then, the free energy $F_N$ defined in Definition \ref{def:partition} converges to
\beq
	F_0(\beta) :=
	\begin{cases}
	\beta^2 & \text{ if } 0<\beta \leq 1/2 \\
	2\beta - \frac{\log (2\beta) + \frac{3}{2}}{2} & \text{ if } \beta > 1/2
	\end{cases}
\eeq
in probability as $N\to \infty$.
\end{thm}

Note that $F_0(\beta)$ is a $C^2$-function of $\beta$; see Theorem 2.9 in \cite{baik2016fluctuations}.

Our next results are about the fluctuations of $F_N$. In the high temperature regime $0 < \beta < 1/2$, we have the following convergence result:
\begin{thm}[High temperature regime] \label{thm:sub}
For $0<\beta<1/2$, there exists a deterministic function $F(\beta)$ such that $F(\beta)=F_0(\beta)+O(N^{-2\phi})$ and
\beq
	N^{\phi+\frac{1}{2}} \left( F_N (\beta)- F(\beta) \right)  \Rightarrow \mathcal{N}\left(0, 2\sigma^2 \beta^4 \right).
\eeq 
\end{thm}
The corresponding result in the low temperature regime $\beta > 1/2$ is as follows:
\begin{thm}[Low temperature regime] \label{thm:sup}
For $\beta>1/2$, there exists a deterministic function $F(\beta)$ such that $F(\beta)=F_0(\beta)+O(N^{-2\phi})$ and random variables $\chi_{TW_1}$ and $\chi_N$ such that
\beq \label{sup relation}
	N^{\min\{2/3,\ (\phi+\frac{1}{2}) \}}\left[F_N(\beta)-F(\beta)-\left(\beta-\frac{1}{2}\right)N^{-2/3}\chi_{TW_1}-\left(\beta-\frac{1}{4}\right)N^{-(\phi+\frac{1}{2})}\chi_N\right] \to 0,
\eeq
where $\chi_{TW_1}$ and $\chi_N$ are asymptotically independent and converge in distribution to the GOE Tracy--Widom distribution and a centered Gaussian distribution with variance $2\sigma^2$, respectively.
\end{thm}

Theorem \ref{thm:third} is an easy consequence of Theorems \ref{thm:sub} and \ref{thm:sup}. 
We remark that the deterministic function $F(\beta)$ can be precisely defined; see \eqref{High T Fbeta} for the case of high temperature, and \eqref{Low T Fbeta} for the case of low temperature.

\subsection{Main idea of proof}

The starting point of the analysis is the following integral representation formula for the partition function, which was first introduced in \cite{kosterlitz1976spherical} and proved in \cite{baik2016fluctuations}.
\begin{prop}[Integral representation] \label{prop:integral}
Let $M$ be an $N \times N$ real symmetric matrix with eigenvalues $\lambda_1 \geq \dots \geq \lambda_N$. Then its partition function defined in Definition \ref{def:partition} satisfies
\beq \begin{split} \label{eq:integral}
	&Z_N \equiv Z_N(\beta) := \int_{S_{N-1}} e^{\beta \langle \bss, M\bss \rangle } \dd \omega_N(\bss) = C_N \int_{\gamma - \ii \infty}^{\gamma + \ii \infty} e^{\frac{N}{2} G(z)} \dd z, \\
	&G(z) := 2\beta z - \frac{1}{N} \sum_{i=1}^N \log (z-\lambda_i),
\end{split} \eeq
for any $\gamma > \lambda_1$, where the $\log$ function is defined in the principal branch and
\[
	C_N = \frac{\Gamma(N/2)}{2\pi \ii (N\beta)^{\frac{N}{2}-1}}.
\]
\end{prop}

See Lemma 1.3 in \cite{baik2016fluctuations} for the proof of Proposition \ref{prop:integral}. With the integral representation formula \eqref{eq:integral}, we find that the free energy can be written as a function of the eigenvalues $\lambda_1, \dots, \lambda_N$. Applying the method of steepest-descent, we choose $\gamma$ to be the critical point satisfying 
\beq \label{eq:gamma}
	G'(\gamma) = 2\beta - \frac{1}{N} \sum_{i=1}^N \frac{1}{\gamma-\lambda_i} = 0.
\eeq
Then we can show that the asymptotic behavior of the free energy is governed by that of $G(\gamma)$. Note that $G'(z)$ is an increasing function of $z$ for $z>\lambda_1$. Since $G'(z) \to 2\beta$ as $z \to \infty$ and $G'(z) \to -\infty$ as $z \searrow \lambda_1$, we find that \eqref{eq:gamma} has a unique solution in $(\lambda_1, \infty)$ for any $\beta>0$.

The exact location of $\gamma$ is random, but with high probability it can be well-approximated by a deterministic point $\wh \gamma$. Heuristically, $G(z)$ can be approximated by
\beq \label{eq:wh_G}
	\wh G(z) = 2\beta z-\int \log (z-k)\, \dd\nu (k)
\eeq
for some deterministic probability measure $\nu$. (See Proposition \ref{cond:regular} for the precise definition of $\nu$.) The equation for $\wh \gamma$ is then
\beq \label{eq:wh_gamma}
	\wh G'(\wh \gamma) = 2\beta -\int \frac{\dd\nu (k)}{\wh \gamma -k} = 0.
\eeq
Unlike \eqref{eq:gamma}, however, the existence of $\wh \gamma$ depends on $\beta$. To see this, let $C_+$ be the upper edge of $\nu$ and define
\beq\label{eq:defbetac}
	\beta_c:=\frac{1}{2} \int \frac{\dd \nu(x)}{C_+ -x}.
\eeq
It is then immediate to see that there exists a (unique) $\widehat{\gamma}=\widehat{\gamma}(\beta)\in (C_+,\infty)$ if and only if $\beta<\beta_c$. We remark that $\wh \gamma \equiv \wh\gamma(\beta)$ is an decreasing function of $\beta\in (0, \beta_c)$ and $\wh\gamma(\beta) \searrow C_+$ as $\beta\nearrow \beta_c$.

From the discussion above, it is natural to conjecture that the $\beta_c$ is the critical inverse temperature for our model. In the high temperature regime, $\beta< \beta_c$, we analyze the behavior of $G(\wh\gamma)$ by first approximating it by $\wh G(\wh \gamma)$. We then apply the result on the linear spectral statistics for sparse matrices to prove Theorem \ref{thm:sub}. (See Proposition \ref{cond:Linear Statistics}.) On the other hand, the critical point $\gamma$ approaches $\lambda_1$ in the low temperature regime, $\beta > \beta_c$, and the behavior of $G(\gamma)$ is affected more by that of $\lambda_1$. Thus, in order to prove Theorem \ref{thm:sup}, it requires to understand the fluctuation of $\lambda_1$. In Proposition \ref{cond:Limit of the Largest Eigenvalue}, we provide a precise description on the asymptotic behavior of $\lambda_1$.
In the analysis involving the method of steepest-descent, for both the high and the low temperature regimes, we use the estimate on the location of the eigenvalues known as the rigidity. As discussed in Introduction, we apply improved rigidity estimates where the randomness originated from $\caZ$ is also taken into consideration. (See Proposition \ref{cond:rigidity} for a more precise statement on the rigidity.)

\section{Preliminaries} \label{sec:prelim}

In this section, we collect previously known results on sparse random matrices that will be used in the proof of the main results, Theorems \ref{thm:sub} and \ref{thm:sup}. These results are typically high-probability estimates for which we use the following notions:

\begin{defn}[Stochastic domination]\label{Stochastic domination}
Let
\beq
	X=(X^{(N)}(u):N\in\mathbb{N}, u\in U^{(N)}),\quad Y=(Y^{(N)}(u):N\in\mathbb{N}, u\in U^{(N)})
\eeq
be two families of random variables, where $Y^{(N)}(u)$ are nonnegative and $U^{(N)}$ is a possibly $N$-dependent parameter set. We say that $X$ is stochastically dominated by $Y$, uniformly in $u$, if for all small $\epsilon>0$ and large $D>0$ we have
\beq
	\sup_{u\in U^{(N)}}\mathbb{P}\left[|X^{(u)}(u)|>N^{\epsilon}Y^{(N)}(u)\right]\le N^{-D}
\eeq
for large enough $N\ge N_0(\epsilon, D)$. If $X$ is stochastically dominated by $Y$, uniformly in $u$, we use the notation $X\prec Y$, or, equivalently $X=O_\prec (Y)$.
\end{defn}

\begin{defn}[High probability event]\label{High Probability Event}
We say that $N$-dependent event $\Omega_N$ holds with high probability if, for any given $D>0$, there exists $N_0>0$ such that
\beq
	\mathbb{P}(\Omega_N^c)\le N^{-D}
\eeq
for any $N>N_0$.
\end{defn}

For a sparse matrix $M$, denote by $\lambda_1\ge\lambda_2\ge \cdots \ge \lambda_N$ the eigenvalues of $M$ and $\nu_N:=N^{-1}\sum_{j=1}^N \delta_{\lambda_j}$ the empirical spectral measure of $M$. While it is well-known that $\nu_N$ converges to the semicircle measure $\nu_{sc}$ in the large $N$ limit, for our purpose we need a more refined measure $\nu$ that approximates $\nu_N$ better than $\nu_{sc}$ for finite $N$, as introduced in \cite{lee2018local} and developed in \cite{lee2021higher}. The existence of such a measure and its properties are collected in the following proposition:

\begin{prop}\label{cond:regular}
There exists a deterministic probability measure $\nu$ that satisfies the following properties:

\begin{itemize}
\item The measure $\nu$ is supported on an interval $[C_-, C_+]$ and is positive on $(C_-, C_+)$.

\item The measure $\nu$ is absolutely continuous (with respect to the Lebesgue measure) and $\frac{\dd \nu}{\dd x}$ exhibits square root decay at the upper edge, i.e.,
\beq\label{eq:sqrtb}
 \frac{\dd \nu}{\dd x}(x) = s_{\nu} \sqrt{C_+-x} \left( 1 + O(C_+ -x) \right) \quad \text{ as } x \nearrow C_+
\eeq
for some $s_{\nu} > 0$.
\item The Stieltjes transform $m(z)$ of $\nu$ is a solution to the polynomial equation
\beq
	P_z(m):=1+zm+m^2+(N\Sigma^2)m^4=0.
\eeq 
$\Sigma$ is defined in \eqref{eq:Sigma}. \eqref{eq:sigma} implies $N\Sigma^2=\sigma^2N^{-2\phi}+o(N^{-2\phi})$. (See Remark 2.9 in \cite{lee2021higher} for detail.)
\item The upper and the lower edges $C_+$ and $C_-$ satisfy $C_+=-C_-=2+N\Sigma^2+O(N^{-4\phi})$.
\end{itemize}
\end{prop}
See Theorems 2.4 and 2.9 in \cite{lee2018local} and Lemmas 2.10 in \cite{he2021fluctuations} for the proof of Proposition \ref{cond:regular}.

Since the measure $\nu$ depends on $N$, it is often required to estimate the difference between the integrals with the measure $\nu$ and those with the semicircle measure $\nu_{sc}$.  For this purpose, we use the following result on the difference between the measures:
\begin{prop}[Difference of measure] \label{measure difference}
Let $m(z)$ and $m_{sc}(z)$ be the Stieltjes transform of the measures $\nu$ and $\nu_{sc}$, respectively. Then for any $z$,
\[
	|m(z) - m_{sc}(z)| = O(N^{-\phi}).
\]
Moreover, for $-2\le x \le 2$, there is a positive constant $C$ independent of $N$ satisfying
\[
	|\dd\nu(x)-\dd\nu_{sc}(x)|\le\frac{CN^{-2\phi}}{\sqrt{4-x^2}}|\dd x|.
\]
\end{prop}
We prove Proposition \ref{measure difference} in Appendix \ref{Proof Measure Difference}.

One of the most crucial inputs for the analysis of the free energy in \cite{baik2016fluctuations} is the rigidity of the eigenvalues, which provides a high-probability bound for the displacement of the random eigenvalues $\lambda_1, \dots, \lambda_N$ from deterministic positions predicted from the limiting measure, known as the classical locations. If $\phi = 1/2$, i.e., $M$ is a Wigner matrix, the classical location $\gamma_{sc, k}$ of the $k$-th eigenvalue associated with the semicircle law can be defined through
\beq \label{eq:sc_classical}
	\int_{\gamma_{sc, k}}^{\infty} \dd \nu_{sc} = \frac{1}{N}\left(k-\frac{1}{2}\right)
\eeq
and the rigidity of the eigenvalues is the statement
\beq \label{eq:Wigner_rigidity}
	|\lambda_k-\gamma_{sc,k}|\prec\displaystyle{\frac{1}{N^{2/3}\min\left\{k, N-k+1\right\}^{1/3}}};
\eeq
see, e.g., Theorem 2.2 in \cite{erdHos2012rigidity}.
For a sparse matrix, however, the estimate based on the semicircle law as in \eqref{eq:Wigner_rigidity} deteriorates as the sparsity parameter $q$ decreases, and can be improved by considering the random variable $\caZ$ in \eqref{eq:caZ}. For the proof of our main results, we use the following rigidity result:

\begin{prop}[Rigidity of eigenvalues] \label{cond:rigidity}
Let $\gamma_k$ be the classical location associated with the probability measure $\nu$ in Proposition \ref{cond:regular}, defined by
\beq\label{eq:classicallocationdef}
	\int_{\gamma_k}^{\infty} \dd \nu = \frac{1}{N}\left(k-\frac{1}{2}\right).
\eeq
For a positive integer $k \in [1, N]$, let $\hat k := \min \{ k, N+1-k \}$. Then, for any $k = 1, \dots,  N$,
\beq \label{rigidity}
	\left|\lambda_k - \gamma_k-\frac{\gamma_{sc,k}}{2}\mathcal{Z}\right| \prec \hat k^{-1/3} N^{-2/3}+N^{-(\frac{1}{2}+3\phi)}.
\eeq
\end{prop}

See Theorem 1.6 in \cite{huang2022edge} and Lemma 2.12 in \cite{he2021fluctuations} for the proof of Proposition \ref{cond:rigidity}.

We remark that the classical locations $\gamma_k$ and $\gamma_{sc,k}$ satisfy 
\beq\label{eq:classicalfromC}
	C^{-1} k^{2/3} N^{-2/3} \leq |C_+ - \gamma_k| \leq C k^{2/3} N^{-2/3}
\eeq
and
\beq\label{eq:classicalfromD}
	C^{-1} k^{2/3} N^{-2/3} \leq |2 - \gamma_{sc,k}| \leq C k^{2/3} N^{-2/3}
\eeq
for some constant $C > 1$ independent of $N$. 
\begin{rem}
In some occasions, the classical locations $\gamma_k$ and/or $\gamma_{sc,k}$ are defined without the constant $\frac{1}{2}$ in \eqref{eq:sc_classical} and/or \eqref{eq:classicallocationdef}. The change of the classical locations due to the constant $\frac{1}{2}$ is negligible in all proofs throughout the paper.
\end{rem}

For the analysis of the free energy in the high temperature regime, it is required to understand the behavior of the linear spectral statistics with a logarithmic function. We introduce a more general statement for the linear spectral statistics as follows:

\begin{prop}[Linear Spectral Statistics] \label{cond:Linear Statistics}
Let $\varphi$ be a smooth function on the open set containing $\mathrm[{-2, 2}]$ such that there is a constant $\alpha$ satisfying $\left|\frac{\varphi^{(n)}(x)}{n!}\right|\le \alpha^{n+1}$ for all differentiation order $n\ge 0$ and $x\in [C_-,C_+]$.
Suppose
\beq
	\int_{-2}^2\varphi(x)\frac{2-x^2}{\sqrt{4-x^2}}\, \dd{x}\neq 0.
\eeq
Then,
\beq \label{Linear Statistics}
	\frac{q}{\sqrt{N}}\sum_{i=1}^N\varphi(\lambda_i)-\mathbb{E}\left[\frac{q}{\sqrt{N}}\sum_{i=1}^N\varphi(\lambda_i)\right]
\eeq
converges to a centered Gaussian random variable with variance
\beq \label{Variance Linear}
	V[\varphi] :=\frac{\sigma^2}{2\pi^2}\left(\int_{-2}^2\varphi(x)\frac{2-x^2}{\sqrt{4-x^2}}\, \dd x \right)^2.
\eeq
(See \eqref{eq:sigma} for the definition of $\sigma$.) Moreover,
\beq \label{eq:LSS_mean}
	\begin{split}
		\mathbb{E}\left[\frac{1}{N}\sum_{i=1}^{N}\varphi(\lambda_i)\right]&=\int_{C_-}^{C_+}\varphi(x)\, \dd{\nu(x)}+O_\prec \left(N^{-1}+N^{-(\frac{1}{2}+3\phi)}\right)\\&=\int_{-2}^{2} \varphi(x)\, \dd{\nu_{sc}(x)}+O(N^{-2\phi}).
	\end{split}
\eeq
\end{prop}

Proposition \ref{cond:Linear Statistics} was proved in \cite{shcherbina2012central}. In Section \ref{sec:LSS}, we will provide a shorter proof of Proposition \ref{cond:Linear Statistics}. We remark that Equation \eqref{eq:LSS_mean} remains valid even when $\varphi$ depends on $N$ as long as $\varphi$ satisfies the condition.

In the low temperature regime, the largest eigenvalue has the dominant role in the fluctuation of the free energy. As discussed in Introduction, it is required to consider both the Tracy--Widom fluctuation and the Gaussian fluctuation to correctly describe the fluctuation of the largest eigenvalue of a sparse matrix. We have the following proposition:

\begin{prop}[Limit of the largest eigenvalue]\label{cond:Limit of the Largest Eigenvalue}
Let $\lambda_1$ be the largest eigenvalue of a sparse matrix $M$. Recall that $C_+$ is the upper edge of the deterministic law in Proposition \ref{cond:regular}. We have
\[
	\lambda_1-C_+= N^{-2/3}\chi_{TW_1} + \caZ + \caE,
\]	
where the random variables where $\chi_{TW_1}$ and $\caE$ satisfy the following:
\begin{itemize}
\item $\chi_{TW_1}$ converges in distribution to the GOE Tracy--Widom law,
\item $\caE \ll N^{-2/3}$ and $\caE \ll N^{-(\phi+\frac{1}{2})}$ with high probability, and
\item $\chi_{TW_1}$ and $N^{2/3}\mathcal{Z}$ are asymptotically independent.
\end{itemize}
\end{prop}

Proposition \ref{cond:Limit of the Largest Eigenvalue} implies that the fluctuation of $\lambda_1$ is given by the GOE Tracy--Widom law if $\phi > 1/6$ and by a Gaussian law if $\phi < 1/6$. In the critical case $\phi = 1/6$, it is given as the convolution of the GOE Tracy--Widom law and the Gaussian law.

\begin{rem}[Notational remark]
Throughout the paper, we use $C$ or $c$ to denote a constant independent on $N$. Even though the constant may vary from place to place, we use the same notation $C$ or $c$ as long as it does not depend on $N$ for the convenience. We use the shorthand notation $\sum_i := \sum_{i=1}^N$.
\end{rem}

\section{High Temperature Case} \label{sec:high}

In this section, we prove Theorem \ref{thm:sub}. Recall that we defined $G(z)$ in Proposition \ref{prop:integral} and $\wh G(z)$ in \eqref{eq:wh_G}, respectively, as
\beq
	G(z)=2\beta z-\frac{1}{N}\sum_i\log (z-\lambda_i), \quad \widehat{G}(z)=2\beta z-\int_{C_-}^{C_+}\log (z-k)\, \dd\nu (k)
\eeq
where $\nu$ is the (deterministic) probability measure in Proposition \ref{cond:regular}. We use the following lemma to control the derivatives of $G$ and $\wh G$.

\begin{lem} 
\label{High T Lemma}
Fix $\delta>0$.
\begin{enumerate}
\item[(i)] We have
\beq \label{eq:High T with epsilon}
	G'(z)-\wh G'(z)\prec N^{-(\frac{1}{2}+\phi)}
\eeq
uniformly in $z\ge C_++\delta$.
\item[(ii)] For each $\ell =0, 1, 2, \dots,$ the derivative $G^{(\ell)}(z)=O(1)$ uniformly in $z \in \mathbb{C}\setminus B_\delta$ with high probability where $B_\delta=\left\{x+iy:C_--\delta<x<C_++\delta, -\delta<y<\delta\right\}$.
\end{enumerate}
\end{lem}

\begin{proof}
Recall the definition of the classical location $\gamma_k$ in \eqref{eq:classicallocationdef}. We introduce a function
\beq \label{eq:wt_G}
	\wt G(z) :=2\beta z-\displaystyle{\frac{1}{N}}\sum_{k=1}^N \log (z-\gamma_k),
\eeq
which bridges $G(z)$ and $\wh G(z)$.
From the rigidity estimate in Proposition \ref{cond:rigidity},
\beq \begin{split} \label{High T Lemma 1-1}
	|G'(z)-\wt G'(z)| &=\left| \frac{1}{N}\sum_i\frac{(\lambda_i-\gamma_i)}{(z-\lambda_i)(z-\gamma_i)} \right| \le \frac{C}{N} \sum_i|\lambda_i-\gamma_i| \\
	& \prec \frac{1}{N} \left(1+|\mathcal{Z}| \sum_i \frac{|\gamma_{sc,i}|}{2} +N^{\frac{1}{2}-3\phi}\right).
\end{split} \eeq
Further, since $\sum_i |\gamma_{sc,i}| = O(N)$ and $\mathcal{Z}\prec N^{-(\frac{1}{2}+\phi)}$, we find from \eqref{High T Lemma 1-1} that
\beq \label{High T Lemma 1-2}
	|G'(z)-\wt G'(z)|\prec N^{-(\frac{1}{2}+\phi)}
\eeq
uniformly in $z\ge C_++\delta$. 

Now, we argue that 
\beq\label{High T Lemma 1-3}
	|\wt G'(z)-\wh G'(z)|=\left|\frac{1}{N}\sum_i \frac{1}{z-\gamma_i}-\int_{C_-}^{C_+}\frac{\dd\nu (k)}{z-k}\right|\le \frac{C}{N}.
\eeq
To do this, we define $\wh{\gamma}_j$ as
\beq\label{def:whgamma}
	\int_{\wh{\gamma}_j}^\infty \dd\nu (k)=\frac{j}{N}, \quad j=1, 2, \ldots, N
\eeq
and $\wh{\gamma}_0=C_+$. Then, we have for $i=2, 3, \ldots, N-1$ that $\wh{\gamma}_i\le\gamma_i\le\wh{\gamma}_{i-1}$ and 
\beq
	\int_{\wh{\gamma}_{i+1}}^{\wh{\gamma_{i}}}\frac{\dd\nu (k)}{z-k}\le \frac{1}{N(z-\gamma_i)}\le\int_{\wh{\gamma}_{i-1}}^{\wh{\gamma}_{i-2}}\frac{\dd\nu (k)}{z-k}
\eeq
uniformly for all $z\ge C_++\delta$. \eqref{High T Lemma 1-3} can be obtained by summing over $i$ and using the simple estimates of $\frac{1}{N(z-\gamma_i)}=O(N^{-1})$ and $\int_{\wh{\gamma}_i}^{\wh{\gamma}_{i-1}}\frac{d\nu (k)}{z-k}=O(N^{-1})$ uniformly for $z\ge C_++\delta$. Combining \eqref{High T Lemma 1-2} and \eqref{High T Lemma 1-3}, we find that \eqref{eq:High T with epsilon} holds. This proves the first part of the lemma.

The proof of the second part of the lemma is straightforward from the formulas
\beq \begin{split}
&G(z)=2\beta z-\displaystyle{\frac{1}{N}}\sum_i\log (z-\lambda_i), \quad G'(z)=2\beta-\displaystyle{\frac{1}{N}}\sum_i\displaystyle{\frac{1}{z-\lambda_i}}, \\
&G^{(\ell)}(z)=\displaystyle{\frac{(-1)^\ell(\ell-1)!}{N}}\sum_i\displaystyle{\frac{1}{(z-\lambda_i)^\ell}} \quad (\ell \geq 2)
\end{split} \eeq
and the rigidity estimate in Proposition \ref{cond:rigidity}. This concludes the proof of the lemma.
\end{proof}

We next provide several estimates on $\wh \gamma$, the deterministic counterpart of $\gamma$.
\begin{lem} \label{High T Cor 2}
Define $\wh \gamma$ by
\beq \label{eq:wh_gamma_def}
	\wh G'(\wh \gamma)=2\beta -\int_{C_-}^{C_+} \frac{\dd\nu (x)}{\wh \gamma-x} = 0.
\eeq
Then, for $\beta< 1/2$,
\beq \label{eq:wh_gamma_approx}
	\wh \gamma = 2\beta + \frac{1}{2\beta} + O(N^{-2\phi})
\eeq
and
\beq\label{High T Cor 2-1}
	|\gamma-\wh\gamma|\prec N^{-(\frac{1}{2}+\phi)}.
\eeq
\end{lem}

\begin{proof}
Suppose that $\beta < \beta_c$. From the definition of $\beta_c$ in \eqref{eq:defbetac} and the fact that $\wh G'(z)$ is an increasing function of $z$ for $z>C_+$, we find that $\wh \gamma$ is uniquely defined by \eqref{eq:wh_gamma_def}.

To prove \eqref{eq:wh_gamma_approx}, we notice that for $\beta < \frac{1}{2}$
\[
	\int_{-2}^{2}\frac{\dd\nu_{sc}(x)}{2\beta + \frac{1}{2\beta} -x} = -m_{sc}(2\beta + \frac{1}{2\beta}) = 2\beta,
\]
which shows that $\wh\gamma = 2\beta + \frac{1}{2\beta}$ if $\nu = \nu_{sc}$. To estimate the difference $|\gamma - \wh \gamma|$, we first approximate the integral in the right-hand side of \eqref{eq:wh_gamma_approx} by
\beq
	\int_{C_-}^{C_+}\frac{\dd\nu}{\widehat{\gamma}-x}=\int_{-2}^{2}\frac{\dd\nu}{\widehat{\gamma}-x}+O(N^{-2\phi}),
\eeq
which can be shown from Proposition \ref{cond:regular} and the fact that $|\wh \gamma-x|$ is uniformly bounded away from $0$ for $x \in [C_-,C_+]$. We thus find from Proposition \ref{measure difference} that
\beq
	\left| \int_{-2}^{2}\frac{\dd\nu}{\widehat{\gamma}-x} - \int_{-2}^{2} \frac{\dd\nu_{sc}}{\widehat{\gamma}-x} \right| \leq CN^{-2\phi} \int_{-2}^{2}\frac{\dd x}{(\widehat{\gamma}-x)\sqrt{4-x^2}} = O(N^{-2\phi})
\eeq
since the integral is convergent with $|\wh{\gamma}-x|$ has order 1 for $-2\le x\le 2$ uniformly. The first part of the lemma follows now from that
\beq
	\int_{-2}^2 \left(\frac{1}{2\beta + \frac{1}{2\beta} -x}-\frac{1}{\widehat{\gamma}-x}\right)\dd\nu_{sc}(x) =\int_{-2}^2 \frac{\widehat{\gamma}- 2\beta - \frac{1}{2\beta}}{(\widehat{\gamma}-x)(2\beta + \frac{1}{2\beta}-x)}\dd\nu_{sc}(x) =O(N^{-2\phi}).
\eeq

We then prove the second part of the lemma. Since $\wh \gamma >C_+$, we can choose $\delta\in(0,(\widehat{\gamma}-C_+)/2)$ and apply Lemma \ref{High T Lemma} to find that for every small $\epsilon>0$, there exists a constant $C_{\epsilon}>0$ such that
\beq
	G'(\widehat{\gamma}\pm C_{\epsilon} N^{-(\frac{1}{2}+\phi)+\epsilon})=\widehat{G}'(\widehat{\gamma}\pm C_{\epsilon} N^{-(\frac{1}{2}+\phi)+\epsilon})+O_\prec(N^{-(\frac{1}{2}+\phi)}).
\eeq
Let $P:=C_{\epsilon} N^{-(\frac{1}{2}+\phi)+\epsilon}$. Now, since $\widehat{G}'(\widehat{\gamma})=0$ and $\widehat{G}'''(z)=O(1)$ for $z$ near $\widehat{\gamma}$, the Taylor expansion of $\widehat{G}$ implies that 
\beq
	G'(\widehat{\gamma}\pm P)=\pm P\widehat{G}''(\widehat{\gamma})+O(P^2)+O_\prec(PN^{-\epsilon}).
\eeq
Noting that $\widehat{G}''(\widehat{\gamma})>0$, this shows that 
\beq
	G'(\widehat{\gamma}+P)>0, \quad G'(\widehat{\gamma}-P)<0
\eeq
with high probability. Since $G'(z)$ is an increasing function of $z$, $|\gamma-\widehat{\gamma}|\le P$ with high probability for every small $\epsilon>0$, which proves the desired lemma for $\beta > \beta_c$.

It remains to prove that $\beta_c = \frac{1}{2} + o(1)$. Let
\beq\label{stieltjes gamma semicircle}
	\beta_{c,sc}:=\frac{1}{2} \int_{-2}^{2} \frac{\dd \nu_{sc}(x)}{2 -x}.
\eeq
Recall that $\nu_{sc}$ is the semicircle measure and $m_{sc}(z)$ is the Stieltjes transform of $\nu_{sc}$. Then, it is well-known that $1+zm_{sc}(z)+m_{sc}(z)^2 = 0$ and it is thus straightforward to see that 
\[
	\beta_{c, sc} = -\frac{1}{2} m_{sc}(2) = \frac{1}{2}.
\]
Since
\[
	\beta_c =\frac{1}{2} \int \frac{\dd \nu(x)}{C_+ -x} = -\frac{1}{2} m(C_+),
\]
it can be readily proved from Lemma \ref{measure difference} that 
\[
	\beta_c = -\frac{1}{2} m_{sc}(C_+) + O(N^{-\phi}) = \beta_{c, sc} + O(N^{-\phi})
\]	
where we also used the facts that $C_+=2+O(N^{-2\phi})$ and $m_{sc}(z)=(-z+\sqrt{z^2-4})/2$. 
Hence, 
\beq\label{beta}
	\beta_c=\frac{1}{2}+O(N^{-\phi})
\eeq
and it finishes the proof of the desired lemma.
\end{proof}

As a simple corollary to Lemma \ref{High T Cor 2}, we also have the following estimates.

\begin{lem}\label{High T Cor 3}
Suppose that $\beta < 1/2$. Then,
\beq \label{High T Cor 3-1}
	G(\gamma)=G(\widehat{\gamma})+O_\prec(N^{-(1+2\phi)}).
\eeq
\end{lem}

\begin{proof}
From the Taylor expansion, the definition of $\gamma$, and Lemma \ref{High T Lemma} (ii),
\beq
	G(\widehat{\gamma})=G(\gamma)+G'(\gamma)(\widehat{\gamma}-\gamma)+O_\prec(|\widehat{\gamma}-\gamma|^2)=G(\gamma)+O_\prec(|\widehat{\gamma}-\gamma|^2).
\eeq
It is then easy to conclude that Lemma \ref{High T Cor 3} holds from the results of Lemma \ref{High T Cor 2}.
\end{proof}

We are now ready to prove our main result on the high temperature regime, Theorem \ref{thm:sub}.

\begin{proof}[Proof of Theorem \ref{thm:sub}]
Applying the method of steepest-descent to the integral in \eqref{eq:integral}, we can find that
\beq
	\int_{\gamma-\ii \infty}^{\gamma+\ii \infty} e^{\frac{N}{2}G(z)}\, \dd z= \frac{\ii e^{\frac{N}{2}G(\gamma)}}{\sqrt{N}} \sqrt{\frac{4\pi}{G''(\gamma)}} \left(1+O_\prec (N^{-1}) \right).
\eeq
(See the proof of Lemma 5.4, especially Equation (5.22), in \cite{baik2016fluctuations}, for the detail.)
By Taylor expanding the exponential function, from Lemma \ref{High T Cor 3}, we get
\beq \begin{split}
	e^{\frac{N}{2}G(\gamma)} &=e^{\frac{N}{2}G(\widehat{\gamma})}(1+O_\prec(N^{-2\phi})).
\end{split} \eeq
This shows that
\beq
	\int_{\gamma-\ii\infty}^{\gamma+\ii\infty} e^{\frac{N}{2}G(z)}\, \dd z=\mathrm{i}e^{\frac{N}{2}G(\widehat{\gamma})} \sqrt{\frac{4\pi}{NG''(\gamma)}} (1+O_\prec(N^{-2\phi})).
\eeq
From Stirling's formula, we have for the constant $C_N$ in \eqref{eq:integral} that
\beq\label{eq:C_N}
	C_N=\frac{\Gamma(N/2)}{2\pi\mathrm{i}(N\beta)^{N/2-1}}=\frac{\sqrt{N}\beta}{\mathrm{i}\sqrt{\pi}(2\beta e)^{N/2}}(1+O(N^{-1})).
\eeq
Hence, we find from Proposition \ref{prop:integral}
\[
	F_N=\frac{1}{N}\log{Z_N}= \frac{1}{2}(G(\widehat{\gamma})-1-\log{2\beta})+\frac{1}{N}\left(\log{2\beta}-\frac{1}{2} \log G''(\gamma) \right)+O_\prec(N^{-1-2\phi}).
\]
Further, since
\[
	G''(\gamma) = -\frac{1}{N} \sum_i \frac{1}{(\gamma - \lambda_i)^2}
\]
and $|\gamma - \lambda_i|$ is uniformly bounded away from $0$ with high probability, we find that $\log G''(\gamma)/N = O_{\prec}(N^{-1})$. Thus,
\beq\label{Free Energy}
	F_N = \frac{1}{2}(G(\widehat{\gamma})-1-\log{2\beta}) + O_\prec(N^{-1}).
\eeq

To prove the fluctuation of the right-hand side of \eqref{Free Energy}, we apply Proposition \ref{cond:Linear Statistics} as follows: We consider a smooth, compactly supported function $\varphi$ such that $\varphi(x):= \log (\wh \gamma - x)$ for $x \in [C_-, C_+]$. Applying Proposition \ref{cond:Linear Statistics} with $\varphi$, we find that $q\sqrt{N}(G(\wh \gamma) - \E[G(\widehat{\gamma})])$ converges to a Gaussian in distribution with the variance
\beq \label{eq:variance_integral}
	\frac{\sigma^2}{2\pi^2} \left(\int_{-2}^2\log (2\beta + \frac{1}{2\beta} -x)\frac{2-x^2}{\sqrt{4-x^2}}\, \dd x \right)^2 = 8 \sigma^2 \beta^4
\eeq where we also use \eqref{eq:wh_gamma_approx} to approximate $\wh \gamma$ by $(2\beta + \frac{1}{2\beta})$. 
(See Appendix \ref{subsec:detail1} for the details of the approximation and evaluation of the integral \eqref{eq:variance_integral}.) We thus find that
\[
	q\sqrt{N} \left( F_N - \frac{1}{2} \big( \E[G(\widehat{\gamma})]-1-\log{2\beta} \big) \right)
\]
converges in distribution to the centered Gaussian with the variance $2 \sigma^2 \beta^4$.

It remains to estimate $\E[G(\widehat{\gamma})]$. From \eqref{eq:LSS_mean},
\[
	\E[G(\widehat{\gamma})] = 2\beta \wh \gamma - \E \left[ \frac{1}{N} \sum_i \varphi(\lambda_i) \right] = 2\beta \wh \gamma - \int_{C_-}^{C_+}\log (\widehat{\gamma}-x)\, \dd\nu (x) + O_\prec \left(N^{-1}+N^{-(\frac{1}{2}+3\phi)}\right).
\]
Thus, letting
\beq\label{High T Fbeta}
	F(\beta):=\left(\beta\widehat{\gamma}-\frac{1+\log (2\beta)}{2}\right)-\frac{1}{2}\int_{C_-}^{C_+}\log (\widehat{\gamma}-x)\, \dd\nu (x),
\eeq
we can conclude that $N^{\phi+\frac{1}{2}}(F_N-F(\beta))$ converges in distribution to the centered Gaussian with the variance $2 \sigma^2 \beta^4$. Lastly, we find from Proposition \ref{measure difference} that $F(\beta) = F_0(\beta) + O(N^{-2\phi})$. (See Appendix \ref{subsec:detail2} for more detail on the evaluation of $F(\beta)$.) This concludes the proof of Theorem \ref{thm:sub}.
\end{proof}


\section{Low Temperature Case} \label{sec:low}

In this section, we prove Theorem \ref{thm:sup}. The first step of the proof is an estimate on the difference between the critical point $\gamma$ and the largest eigenvalue $\lambda_1$. While this difference is approximately of order $N^{-1}$ for the Wigner case (see Lemma 6.1 in \cite{baik2016fluctuations}), it depends on the sparsity in our model. Throughout Section \ref{sec:low}, we let $t:=\min\{\phi+\frac{1}{2}, \frac{2}{3} \}$. Note that $\mathcal{Z}\prec N^{-t}$.

\begin{lem}\label{Low T Lemma 1}
Let $\beta>\beta_c$. Then,
\begin{equation}
	\frac{1}{3\beta N}\le \gamma-\lambda_1\prec\frac{1}{N^{3t/2}}.
\end{equation}
\end{lem}

\begin{proof}
From the definition of $G(z)$, we find for any $z>\lambda_1$ that
\beq
	G'(z)=2\beta-\frac{1}{N}\sum_i\frac{1}{z-\lambda_i}<2\beta-\frac{1}{N}\frac{1}{z-\lambda_1}.
\eeq
Thus, $G'(\lambda_1+\frac{1}{3\beta N})<0$. Since $G'(z)$ is an increasing function of $z \in (\lambda_1, \infty)$, this proves the first inequality of the desired lemma. Further, it also shows that the second inequality of the desired lemma holds if $G'(\lambda_1+N^{-\frac{3t}{2}+4\epsilon})>0$ with high probability for any sufficiently small $\epsilon>0$.

Set $z=\lambda_1+N^{-\frac{3t}{2}+4\epsilon}$. We decompose $G(z)$ as
\beq
	G'(z)=2\beta-\frac{1}{N}\sum_{i=1}^{N^{1-\frac{3t}{2}+3\epsilon}}\frac{1}{z-\lambda_i}-\frac{1}{N}\sum_{i=N^{1-\frac{3t}{2}+3\epsilon}+1}^{N-N^{1-\frac{3t}{2}+3\epsilon}}\frac{1}{z-\lambda_i}-\frac{1}{N}\sum_{i=N-N^{1-\frac{3t}{2}+3\epsilon}+1}^{N}\frac{1}{z-\lambda_i}.
\eeq
For $1\le i\le N^{1-\frac{3t}{2}+3\epsilon}$, 
\beq
	\frac{1}{N}\sum_{i=1}^{N^{1-\frac{3t}{2}+3\epsilon}}\frac{1}{z-\lambda_i}=O(N^{-\epsilon}).
\eeq
For $N^{1-\frac{3t}{2}+3\epsilon}<i\le N-N^{1-\frac{3t}{2}+3\epsilon}$, we have from Proposition \ref{cond:rigidity} that 
\[
	|\lambda_i-\gamma_i|\le N^{-1+\frac{t}{2}} + \frac{|\gamma_{sc,i}|}{2}|\mathcal{Z}|+ N^{-(\frac{1}{2}+3\phi)+\epsilon}
\]
with high probability. We also note that $C_+-\gamma_i\ge C^{-1}N^{-t+2\epsilon}$ from \eqref{eq:classicalfromC}. On the other hand, since
\[
	\lambda_1-\gamma_1\le N^{-\frac{2}{3}+\epsilon} + |\mathcal{Z}| + N^{-(\frac{1}{2}+3\phi)+\epsilon} 
\]
from the rigidity condition, and $C_+-\gamma_1=O(N^{-2/3})$ from \eqref{eq:classicalfromC}, we have 
\[
	\lambda_1=C_+ + O(|\mathcal{Z}| + N^{-\frac{2}{3}+\epsilon}+N^{-(\frac{1}{2}+3\phi)+\epsilon}).
\]
Altogether, we can find
\beq \begin{split}
	&\sum_{i=N^{1-\frac{3t}{2}+3\epsilon}+1}^{N-N^{1-\frac{3t}{2}+3\epsilon}}\frac{1}{\lambda_1+N^{-\frac{3t}{2}+4\epsilon}-\lambda_i}\\
	&=\sum_{i=N^{1-\frac{3t}{2}+3\epsilon}+1}^{N-N^{1-\frac{3t}{2}+3\epsilon}}\frac{1}{C_+-\gamma_i} \big( 1+O(N^{-\frac{2}{3}+t-\epsilon})+O(N^{t-2\epsilon}|\mathcal{Z}|)+O(N^{t-(\frac{1}{2}+3\phi)-\epsilon}) \big).
\end{split} \eeq
In order to estimate the right-hand side, recall the definition of $\wh\gamma_k$ for $k=1,2,\ldots, N$ in \eqref{def:whgamma} and $\wh{\gamma}_0=C_+$. Since $\widehat{\gamma}_k \le\gamma_k \le\widehat{\gamma}_{k-1}$, for $k=2, 3, \ldots, N-1$,
\beq
	\int_{\widehat{\gamma}_{k+1}}^{\widehat{\gamma}_k}\frac{\dd\nu(x)}{C_+-x}\le\frac{1}{N}\frac{1}{C_+-\gamma_k}\le\int_{\widehat{\gamma}_{k-1}}^{\widehat{\gamma}_{k-2}}\frac{\dd\nu (x)}{C_+-x}.
\eeq
Summing these inequalities,
\beq\label{Low T Lemma 1-1} \begin{split}
	\left|\frac{1}{N}\sum_{i=N^{1-\frac{3t}{2}+3\epsilon}+1}^{N-N^{1-\frac{3t}{2}+3\epsilon}}\frac{1}{C_+-\gamma_i}-2\beta_c\right|=\left|\frac{1}{N}\sum_{i=N^{1-\frac{3t}{2}+3\epsilon}+1}^{N-N^{1-\frac{3t}{2}+3\epsilon}}\frac{1}{C_+-\gamma_i}-\int_{C_-}^{C_+}\frac{\dd\nu (x)}{C_+-x}\right|\\ \le C\left(\int_{C_-}^{\widehat{\gamma}_{N-N^{1-\frac{3t}{2}+3\epsilon}-2}}\frac{\dd\nu (x)}{C_+-x}+\int_{\widehat{\gamma}_{N^{1-\frac{3t}{2}+3\epsilon}+2}}^{C_+}\frac{\dd\nu (x)}{C_+-x}\right)=O(N^{-\frac{t}{2}+\epsilon}).
\end{split} \eeq

Lastly, for $N-N^{1-\frac{3t}{2}+3\epsilon}<i\le N$, it is not hard to see that $\lambda_1-\lambda_i\ge1$, and hence
\beq
	\frac{1}{N}\sum_{i=N-N^{1-\frac{3t}{2}+3\epsilon}+1}^{N}\frac{1}{z-\lambda_i}=O(N^{-\frac{3t}{2}+3\epsilon}).
\eeq
In conclusion, we have shown that 
\beq
	G'(\lambda_1+N^{-\frac{3t}{2}+4\epsilon})=2\beta-2\beta_c+o(1)>0.
\eeq
with high probability, which completes the proof of the desired lemma.
\end{proof}

As in the high temperature regime, we will show that the main contribution to the free energy in the integral representation formula \eqref{eq:integral} is due to $G(\gamma)$. Heuristically, it can be approximated as
\[
	G(\gamma) = 2\beta \gamma - \frac{1}{N} \sum_i \log(\gamma - \lambda_i) \approx 2\beta \lambda_1 - \frac{1}{N} \sum_i \left( \log(C_+ - \lambda_i) + \frac{\gamma - C_+}{C_+ - \lambda_i} \right).
\]
For Wigner matrices, the sum involving $\lambda_i$ can then be approximated by the definite integral with the deterministic measure $\nu_{sc}$; we refer to Lemma 6.2 in \cite{baik2016fluctuations} for more detail. For sparse matrices, however, the randomness in the rigidity estimate in Proposition \ref{cond:rigidity} is not negligible and it must be taken into consideration in the approximation of $G(\gamma)$. The precise statement for this approximation is the following lemma.

\begin{lem}\label{Low T Lemma 2}
Let $\beta>\beta_c$. Then,
\beq \begin{split}
	G(\gamma)&=2\beta\lambda_1-\int_{C_-}^{C_+}\log (C_+-x)\dd\nu (x)+\frac{\mathcal{Z}}{2}-2\beta_c(\lambda_1-C_+)+O_\prec(N^{-3t/2}+N^{-(\frac{1}{2}+\frac{5\phi}{3})}).
\end{split} \eeq
Moreover, there is a constant $C_0>0$ independent on $\ell$ such that for all $\ell=2, 3, \ldots$,
\beq
	N^{\frac{3t\ell}{2}-1}\prec\frac{(-1)^\ell}{(\ell-1)!}G^{(\ell)}(\gamma)\prec C_0^\ell N^{\ell-\frac{3t}{2}}.
\eeq
\end{lem}

\begin{proof}
Fix a (small) constant $\epsilon>0$. We begin by decomposing the sum $N^{-1} \sum_i \log (\gamma - \lambda_i)$ as in the proof of Lemma \ref{Low T Lemma 1}. Since $\gamma-\lambda_1\ge\frac{1}{3\beta N}$ from Lemma \ref{Low T Lemma 1},
\beq \label{eq:Low T Lemma 2_small}
	\left|\frac{1}{N}\sum_{i=1}^{N^{1-\frac{3t}{2}+3\epsilon}}\log (\gamma-\lambda_i)\right|\le\frac{CN^{3\epsilon}}{N^{3t/2}}\log N.
\eeq
We also have with high probability that
\beq \label{eq:Low T Lemma 2_large}
	\left|\frac{1}{N}\sum_{i=N-N^{1-\frac{3t}{2}+3\epsilon}+1}^{N}\log (\gamma-\lambda_i)\right|\le\frac{CN^{3\epsilon}}{N^{3t/2}}.
\eeq
We now consider the case $N^{1-\frac{3t}{2}+3\epsilon}<i<N-N^{1-\frac{3t}{2}+3\epsilon}+1$. We notice that
\beq \begin{split} \label{eq:log_gamma}
	&\log (\gamma-\lambda_i)-\log (C_+-\gamma_i)+\frac{\gamma_{sc,i}/2}{C_+-\gamma_i}\mathcal{Z}=\log \left(1+\frac{\gamma-C_+}{C_+-\gamma_i}+\frac{\gamma_i-\lambda_i}{C_+-\gamma_i}\right)+\frac{\gamma_{sc,i}/2}{C_+-\gamma_i}\mathcal{Z}\\&=\frac{\gamma-C_+}{C_+-\gamma_i}+\frac{\gamma_i-\lambda_i+(\gamma_{sc,i}/2)\mathcal{Z}}{C_+-\gamma_i}+O\left(\left(\frac{\gamma-C_+}{C_+-\gamma_i}\right)^2\right)+O\left(\left(\frac{\lambda_i-\gamma_i}{C_+-\gamma_i}\right)^2\right).
\end{split} \eeq

The first error term in \eqref{eq:log_gamma} can be estimated by 
\beq \begin{split}
	\left(\frac{\gamma-C_+}{C_+-\gamma_i}\right)^2 &\prec \frac{N^{-1-6\phi}+N^{-4/3}+N^{-2/3}|\mathcal{Z}|+\mathcal{Z}^2}{i^{4/3}N^{-4/3}}\\
	&\prec (1+N^{\frac{1}{3}-6\phi}+N^{\frac{1}{6}-\phi}+N^{\frac{1}{3}-2\phi})i^{-4/3}=C(1+N^{\frac{1}{3}-2\phi})i^{-4/3},
\end{split} \eeq
where we used the decomposition
\beq\label{gamma-C_+}
	\gamma-C_+=(\gamma-\lambda_1)+(\lambda_1-\gamma_1)+(\gamma_1-C_+)=O_\prec(N^{-2/3}+N^{-(\frac{1}{2}+3\phi)})+O(|\mathcal{Z}|),
\eeq
which can be seen from \eqref{eq:classicalfromC} and Lemma \ref{Low T Lemma 1}. It is straightforward to see that
\beq\label{Low T Lemma 2-1}
	\frac{1}{N}\sum_{i=N^{1-\frac{3t}{2}+3\epsilon}}^{N-N^{1-\frac{3t}{2}+3\epsilon}}i^{-4/3}\le CN^{-\frac{4}{3}+\frac{t}{2}-\epsilon}.
\eeq
Hence,
\beq \label{eq:Low T Lemma 2_mid1}
	\sum_{i=N^{1-\frac{3t}{2}+3\epsilon}}^{N-N^{1-\frac{3t}{2}+3\epsilon}}\left(\frac{\gamma-C_+}{C_+-\gamma_i}\right)^2\prec N^{-\frac{4}{3}+\frac{t}{2}-\epsilon}+N^{-1-2\phi+\frac{t}{2}-\epsilon}\le 2N^{-\epsilon}N^{-3t/2}.
\eeq

For the second error term in \eqref{eq:log_gamma},
\beq \begin{split}
	\left(\frac{\lambda_i-\gamma_i}{C_+-\gamma_i}\right)^2 &\prec\left(\frac{N^{-2/3}\hat{i}^{-1/3}+(|\gamma_{sc,i}/2|)|\mathcal{Z}|+N^{-(\frac{1}{2}+3\phi)}}{i^{2/3}N^{-2/3}}\right)^2 \\
	&\le\left(\frac{N^{-2/3}\hat{i}^{-1/3}+|\mathcal{Z}|+N^{-(\frac{1}{2}+3\phi)}}{i^{2/3}N^{-2/3}}\right)^2.
\end{split} \eeq
Since $\mathcal{Z}\prec N^{-(\phi+\frac{1}{2})}\le N^{-t}$, from \eqref{Low T Lemma 2-1}, we find
\beq \begin{split}
	&\frac{1}{N}\sum_{i=N^{1-\frac{3t}{2}+3\epsilon}}^{N-N^{1-\frac{3t}{2}+3\epsilon}}\left(\frac{\lambda_i-\gamma_i}{C_+-\gamma_i}\right)^2 \\
	&\prec N^{-\frac{4}{3}+ \frac{t}{2}-\epsilon}(N^{-\frac{2}{3}+t-2\epsilon}+N^{\frac{4}{3}}|\mathcal{Z}|^2+N^{\frac{1}{3}-6\phi} +2|\mathcal{Z}|N^{\frac{1}{3}+\frac{t}{2}-\epsilon}+2N^{\frac{5}{6}-3\phi}|\mathcal{Z}|+2N^{-\frac{1}{6}-3\phi+ \frac{t}{2}-\epsilon})\\&\prec N^{-2+\frac{3t}{2}-3\epsilon}+N^{-1-6\phi+ \frac{t}{2}-\epsilon}+2N^{-\frac{3t}{2}-\epsilon}+N^{-1-2\epsilon} =O(N^{-\frac{3t}{2}-\epsilon}).
\end{split} \eeq
We use $\widehat{i}\ge N^{1-\frac{3t}{2}+3\epsilon}$ during the summation.

For the second term in the right hand side of \eqref{eq:log_gamma} containing $\mathcal{Z}$, we consider two different cases. For $N^{1-\frac{3t}{2}+3\epsilon}<i\le N/2$, we use the estimate
\begin{align}
	\frac{|\gamma_i-\lambda_i+(\gamma_{sc,i}/2)\mathcal{Z}|}{C_+-\gamma_i} \prec\frac{i^{-1/3}N^{-2/3}+N^{-(\frac{1}{2}+3\phi)}}{i^{2/3}N^{-2/3}} =\left(i^{-1}+i^{-2/3}N^{\frac{1}{6}-3\phi}\right).
\end{align}
For $N/2<i<N-N^{1-\frac{3t}{2}+3\epsilon}$, we similarly have
\beq\begin{split}
	\frac{|\gamma_i-\lambda_i+(\gamma_{sc,i}/2)\mathcal{Z}|}{C_+-\gamma_i}& \prec \frac{(N-i)^{-1/3}N^{-2/3}+ N^{-(\frac{1}{2}+3\phi)}}{i^{2/3}N^{-2/3}} = \left([i^2(N-i)]^{-1/3}+i^{-2/3}N^{\frac{1}{6}-3\phi}\right)\\ &\le \left(N^{-2/3}(N-i)^{-1/3}+i^{-2/3}N^{\frac{1}{6}-3\phi}\right).
\end{split} \eeq
Hence,
\beq \begin{split} \label{eq:Low T Lemma 2_mid2}
	&\frac{1}{N}\sum_{i=N^{1-\frac{3t}{2}+3\epsilon}}^{N-N^{1-\frac{3t}{2}+3\epsilon}}\frac{|\gamma_i-\lambda_i+(\gamma_{sc,i}/2)\mathcal{Z}|}{C_+-\gamma_i} \\
	& \prec \frac{1}{N}\sum_{i=N^{1-\frac{3t}{2}+3\epsilon}}^{N/2}\left(i^{-1}+i^{-2/3}N^{\frac{1}{6}-3\phi}\right) +\frac{1}{N}\sum_{i=N/2+1}^{N-N^{1-\frac{3t}{2}+3\epsilon}}\left(N^{-2/3}(N-i)^{-1/3}+i^{-2/3}N^{\frac{1}{6}-3\phi}\right)\\ &= O(N^{-1}\log N+N^{-(\frac{1}{2}+3\phi)}).
\end{split} \eeq

Combining \eqref{eq:Low T Lemma 2_small}, \eqref{eq:Low T Lemma 2_large}, \eqref{eq:log_gamma}, \eqref{eq:Low T Lemma 2_mid1}, and \eqref{eq:Low T Lemma 2_mid2}, we obtain
\beq \begin{split} \label{eq:Low T Lemma 2_combine}
	&\left|\frac{1}{N}\sum_i \log (\gamma-\lambda_i) -\frac{1}{N}\sum_{i=N^{1-\frac{3t}{2}+3\epsilon}+1}^{N-N^{1-\frac{3t}{2}+3\epsilon}}\left[\log (C_+-\gamma_i)- \frac{\gamma_{sc,i}/2}{C_+-\gamma_i}\mathcal{Z} +\frac{\gamma-C_+}{C_+-\gamma_i}\right]\right| \\
	&\le N^{-\frac{3t}{2}+4\epsilon}+ N^{-(\frac{1}{2}+3\phi)+2\epsilon}.
\end{split} \eeq

Following the proof of Lemma \ref{Low T Lemma 1}, we can convert the sum involving $\log (C_+-\gamma_i)$ in \eqref{eq:Low T Lemma 2_combine} as
\beq \begin{split}
	&\left|\frac{1}{N}\sum_{i=N^{1-\frac{3t}{2}+3\epsilon}+1}^{N-N^{1-\frac{3t}{2}+3\epsilon}}\log (C_+-\gamma_i)-\int_{C_-}^{C_+}\log (C_+-x)\, d\nu (x)\right|\\ &\le C\left(\int_{C_-}^{\widehat{\gamma}_{N-N^{1-\frac{3t}{2}+3\epsilon}-2}}\log (C_+-x)\, d\nu (x)+\int_{\widehat{\gamma}_{N^{1-\frac{3t}{2}+3\epsilon}+2}}^{C_+}\log (C_+-x)\, d\nu (x)\right)\\ &\le CN^{-\frac{3t}{2}+3\epsilon}+C\int_0^{N^{-t+2\epsilon}}\sqrt{t}\log t\, dt\\ &\le CN^{-\frac{3t}{2}+3\epsilon}\log N.
\end{split} \eeq
We also have
\[
	\left|\frac{1}{N}\sum_{i=1}^{N^{1-\frac{3t}{2}+3\epsilon}}\frac{\gamma_{sc,i}/2}{C_+-\gamma_i}\right| \le \frac{C}{N}\sum_{i=1}^{N^{1-\frac{3t}{2}+3\epsilon}}i^{-2/3}N^{2/3} \le\frac{CN^{\epsilon}}{N^{t/2}}
\]
and
\[ \begin{split}
	\left|\frac{1}{N}\sum_{i=N-N^{1-\frac{3t}{2}+3\epsilon}+1}^{N}\frac{\gamma_{sc,i}/2}{C_+-\gamma_i}\right|&\le \frac{C}{N}\sum_{i=N-N^{1-\frac{3t}{2}+3\epsilon}+1}^{N}i^{-2/3}N^{2/3}\\&\le\frac{C}{N^{1/3}}(N^{1/3}-(N-N^{1-\frac{3t}{2}+3\epsilon})^{1/3})=O(N^{-\frac{3t}{2}+3\epsilon}),
\end{split} \]
which can be used to control the sum involving $\caZ$ in \eqref{eq:Low T Lemma 2_combine}. We proved in \eqref{Low T Lemma 1-1} that
\beq
	\left|\frac{1}{N}\sum_{i=N^{1-\frac{3t}{2}+3\epsilon}+1}^{N-N^{1-\frac{3t}{2}+3\epsilon}}\frac{1}{C_+-\gamma_i}-2\beta_c\right|=O(N^{-\frac{t}{2}+\epsilon}).
\eeq
Hence
\beq
	\left|(\gamma-C_+)\left[\frac{1}{N}\sum_{i=N^{\frac{1}{2}-\phi-6\epsilon}+1}^{N-N^{1-\frac{3t}{2}+3\epsilon}}\frac{1}{C_+-\gamma_i}-2\beta_c\right]\right|=O_\prec(N^{-3t/2})
\eeq
since $\gamma-C_+ \prec N^{-t}$, shown in \eqref{gamma-C_+}.

We thus conclude that
\beq \begin{split}
	&\left|\frac{1}{N}\sum_i\log (\gamma-\lambda_i)-\int_{C_-}^{C_+}\log (C_+-x)d\nu (x)+\frac{\mathcal{Z}}{N}\sum_i \frac{\gamma_{sc,i}/2}{C_+-\gamma_i}-2\beta_c(\gamma-C_+)\right|\\
	&=O(N^{-\frac{3t}{2}+4\epsilon})+O(N^{-(\frac{1}{2}+3\phi)+2\epsilon}).
\end{split} \eeq
This shows that
\beq 
	G(\gamma)=2\beta\lambda_1-\int_{C_-}^{C_+}\log (C_+-x)\dd\nu (x)+\frac{\mathcal{Z}}{N}\sum_i\frac{\gamma_{sc,i}/2}{C_+-\gamma_i}-2\beta_c(\lambda_1-C_+)\\+O_\prec(N^{-3t/2}+N^{-(\frac{1}{2}+3\phi)}).
 \eeq
It is elementary to prove that
\beq
	\frac{1}{N}\sum_i \frac{\gamma_{sc,i}/2}{C_+-\gamma_i}=\frac{1}{2}+O(N^{-2\phi/3}),
\eeq
which finishes the proof of the first part of the lemma. (See Appendix \ref{subsec:detail3} for the detail.)

We now prove the second part of the lemma. We have
\beq
	G^{(\ell)}(\gamma)=\frac{(-1)^\ell(\ell-1)!}{N}\sum_{i=1}^N\frac{1}{(\gamma-\lambda_i)^\ell}
\eeq
for $\ell=2, 3, \ldots$. For the lower bound, we use Lemma \ref{Low T Lemma 1} to obtain
\beq
	\sum_{i=1}^N\frac{1}{(\gamma-\lambda_i)^\ell}\ge\frac{1}{N}\frac{1}{(\gamma-\lambda_1)^\ell}\ge N^{\frac{3t\ell}{2}-1-4\ell\epsilon}.
\eeq
For the upper bound, we use
\beq
	\frac{1}{N}\sum_{i=1}^{N^{1-\frac{3t}{2}+3\epsilon}}\frac{1}{(\gamma-\lambda_i)^\ell}\le N^{-\frac{3t}{2}+3\epsilon}\frac{1}{(\gamma-\lambda_1)^\ell}\le (3\beta)^\ell N^{\ell-\frac{3t}{2}+3\epsilon}
\eeq
and
\beq \begin{split}
	\frac{1}{N}\sum_{i=N^{1-\frac{3t}{2}+3\epsilon}}^{N}\frac{1}{(\gamma-\lambda_i)^\ell}&\le\frac{1}{N}\sum_{i=N^{1-\frac{3t}{2}+3\epsilon}}^{N}\frac{2^\ell N^{3\epsilon\ell}}{(C_+-\gamma_i)^\ell} \le \frac{2^\ell N^{3\epsilon\ell}}{N}\sum_{i=N^{1-\frac{3t}{2}+3\epsilon}}^{N}\frac{1}{(i^{2/3}N^{-2/3})^\ell}\\
	& \le \frac{2^\ell}{(2/3)\ell-1}N^{(\frac{2}{3}\ell-1)(\frac{3t}{2}-3\epsilon)+3\ell\epsilon}.
\end{split} \eeq
This implies that there is a constant $C_0>0$ independent on $\ell$ such that for all $\ell=2, 3, \ldots$,
\[
	N^{\frac{3t\ell}{2}-1-4\ell\epsilon}\le\frac{(-1)^\ell}{(\ell-1)!}G^{(\ell)}(\gamma)\le C_0^\ell N^{\ell-\frac{3t}{2}+3\epsilon}
\]
with high probability, which proves the desired Lemma.
\end{proof}

\begin{lem}\label{Low T Lemma 3}
	Let $\beta>\beta_c$. Let $\gamma$ be the unique solution to the equation $G'(\gamma)=0$. Then. there exists $K\equiv K(N)$ satisfying $N^{-C}<K<C$ for some constant $C>0$ such that
\beq
	\int_{\gamma-\mathrm{i}\infty}^{\gamma+\mathrm{i}\infty}e^{\frac{N}{2}G(z)}\, dz=\mathrm{i}e^{\frac{N}{2}G(\gamma)}K
\eeq
with high probability.
\end{lem}

\begin{proof}
We follow the proof of Lemma 6.3 in \cite{baik2016fluctuations}, where the desired lemma was proved for the case $\phi=1/2$. While the most parts of the proof remain unchanged even if $\phi<1/2$, there are several changes in the constants appearing in the proof of lower bound of $K$, since Lemma \ref{Low T Lemma 2} in this paper does not exactly match its counterpart in \cite{baik2016fluctuations}. These changes are as follows:
\begin{itemize}
\item Equation (6.47) in \cite{baik2016fluctuations} changes to
\beq
	|\Omega|\le\sum_{j=4}^\infty\frac{C_0^jN^{j-2\phi+3\epsilon}}{N^{4\phi-1-8\epsilon}}|X+\mathrm{i}Y|^{j-1}\le 2C_0^4N^{6\phi-3+11\epsilon}(X^2+Y^2)
\eeq 
for $z=(X+\gamma)+\mathrm{i}Y\in\Gamma^+\cap\bar{B}_{N^{-8+12\phi}}$, for all $N>{(2C_0)}^{\frac{1}{(7-12\phi)}}$. The ball $B_{N^{-2}}$ in Lemma 6.3 in \cite{baik2016fluctuations} is replaced by $B_{N^{-8+12\phi}}$.
\item Equation (6.48) in \cite{baik2016fluctuations} becomes
\beq
	\frac{1}{C_0^2}N^{8\phi-3-15\epsilon}\le\alpha\le C_0^3N^{4-6\phi+11\epsilon}.
\eeq
\item Equation (6.49) in \cite{baik2016fluctuations} becomes
\beq
	N^{6\phi-3+11\epsilon}X^2\ll \alpha X^2\ll X,\quad N^{6\phi-3+11\epsilon}Y^2\ll \alpha Y^2.
\eeq
\item Equation (6.54) in \cite{baik2016fluctuations} becomes
\beq
	G(\gamma)-G(z_2)\ge\frac{1}{8}N^{28\phi-17-8\epsilon}.
\eeq
\item In Equation (6.55) in \cite{baik2016fluctuations}, $B_{N^{-3}}$ is replaced by $B_{N^{-\frac{19}{2}+13\phi}}$. Then, equation (6.55) becomes
\beq
	G(\gamma)-G(z_3)\le2C_0^2N^{-17+24\phi+3\epsilon}
\eeq
for all large $N$ where $C_0$ is the constant in Lemma \ref{Low T Lemma 2}.
\item Equation (6.59) becomes
\beq
	\frac{1}{2}K\ge CN^{-\frac{19}{2}+13\phi}\left[\mathrm{exp}(-C_0^2N^{-16+24\phi+3\epsilon})-\mathrm{exp}(-\frac{1}{16}N^{-16+28\phi-8\epsilon})\right]\ge CN^{-\frac{51}{2}+41\phi-8\epsilon}.
\eeq
\end{itemize}
\end{proof}

We now prove our main result on the low temperature regime, Theorem \ref{thm:sup}.
\begin{proof}[Proof of Theorem \ref{thm:sup}]
Proceeding as in the proof of Theorem \ref{thm:sup}, with Proposition \ref{prop:integral}, Lemma \ref{Low T Lemma 3}, and Equation \eqref{eq:C_N}, we find that
\beq
	Z_N=\frac{\sqrt{N}\beta}{\mathrm{i}\sqrt{\pi}(2\beta e)^{N/2}}e^{\frac{N}{2}G(\gamma)}K(1+O(N^{-1}))
\eeq
with high probability, where $N^{-C}\le K\le C$. Thus, using Lemma \ref{Low T Lemma 2}, we find that
\begin{equation}\label{Low T Free Energy}
	\begin{split}
		F_N=\frac{1}{N}\log {Z_N}&=\frac{1}{2}\left[G(\gamma)-1-\log {(2\beta)}\right]+O(N^{-1}\log N)\\ &=\frac{1}{2}\left[2\beta\lambda_1-2\beta_c(\lambda_1-C_+)-\int_{C_-}^{C_+}
\log (C_+-x)\, d\nu (x)-1-\log (2\beta)\right]\\&\qquad\qquad+\frac{\mathcal{Z}}{4}+O_\prec(N^{-3t/2}+N^{-(\frac{1}{2}+\frac{5\phi}{3})}).
	\end{split}
\end{equation}
Define 
\beq\label{Low T Fbeta}
	F(\beta)=\beta C_+-\frac{1}{2}\left(\int_{C_-}^{C_+}
\log (C_+-x)\, \dd\nu (x)-1-\log (2\beta)\right).
\eeq
As in the high-temperature regime, we have from Proposition \ref{measure difference} that $F(\beta)=F_0(\beta)+O(N^{-2\phi})$. (See Appendix \ref{subsec:detail2} for the detail). Hence, from \eqref{Low T Free Energy},
\beq\label{FN-Fbeta}
	F_N-F(\beta)=(\beta-\beta_c)(\lambda_1-C_+)+\frac{\mathcal{Z}}{4}+O_\prec(N^{-3t/2}+N^{-(1/2+5\phi/3)}).
\eeq
Using Proposition \ref{cond:Limit of the Largest Eigenvalue} with \eqref{FN-Fbeta},
\beq
	F_N-F(\beta)=(\beta-\beta_c)N^{-2/3}\chi_{TW_1}+(\beta-\beta_c+\frac{1}{4})\mathcal{Z}+(\beta-\beta_c)\mathcal{E}+O_\prec(N^{-3t/2}+N^{-(1/2+5\phi/3)}).
\eeq
Recall that $\beta_c = \frac{1}{2} + O(N^{-\phi})$ as shown in \eqref{beta}. Since $N^{\phi+1/2}\mathcal{Z}\to \mathcal{N}(0,2\sigma^2)$, we conclude the proof of the desired theorem. 
\end{proof}


\section{Linear Spectral Statistics}\label{sec:LSS}
In this section, we provide a short proof of Proposition \ref{cond:Linear Statistics} based on the random variable $\mathcal{Z}$. 

As mentioned in Introduction, we may conjecture the asymptotic normality of the linear spectral statistics is due to the dominance of the random variable $\mathcal{Z}$ in the fluctuation of the eigenvalues. 

\begin{proof}[Proof of Proposition \ref{cond:Linear Statistics}]
Recall the rigidity estimate in Proposition \ref{cond:rigidity},
\[
	\left|\lambda_k-\gamma_{k}-\displaystyle{\frac{\gamma_{sc,k}}{2}}\mathcal{Z}\right|\prec N^{-2/3}\hat{k}^{-1/3}+N^{-\left(\frac{1}{2}+3\phi\right)}.
\]
Since $\mathcal{Z}\prec N^{-(\phi+\frac{1}{2})}$, by Taylor expanding $\varphi$,
\beq \label{Linear Statistics 1}
\begin{split}
	\frac{1}{N}\sum_{i=1}^{N}\varphi(\lambda_i)&=\frac{1}{N}\sum_{i=1}^{N}\left\{\varphi(\gamma_i)+\varphi '(\gamma_i)(\lambda_i-\gamma_i)+O(|\lambda_i-\gamma_i|^2)\right\}\\&=\frac{1}{N}\sum_{i=1}^{N}\left[\varphi(\gamma_i)+\varphi '(\gamma_i)(\lambda_i-\gamma_i)\right]+O_\prec(N^{-4/3}+N^{-1-6\phi})+O_\prec(|\mathcal{Z}|^2+N^{-2/3}|\mathcal{Z}|)\\&=\frac{1}{N}\left[\sum_{i=1}^{N}\varphi(\gamma_i)\right]+\frac{1}{N}\left[\sum_{i=1}^{N}\varphi '(\gamma_i)(\lambda_i-\gamma_i)\right]+O_\prec(N^{-4/3}+N^{-(1+2\phi)}).
\end{split}
\eeq
By the mean value theorem, for each $2\le i\le N$ and $\lambda\in [\gamma_{i-1}, \gamma_i]$, there exists a constant $c_\lambda$ such that 
\[
	\varphi(\lambda)=\varphi(\gamma_i)+\varphi'(c_\lambda)(\lambda-\gamma_i).
\]
From the assumption that $\varphi'$ is uniformly bounded,
\[
	|\varphi(\lambda)-\varphi(\gamma_i)|\le \alpha^2(\gamma_{i-1}-\gamma_i).
\]
Plus, by the definition of $\gamma_i$,\ $\displaystyle{\int_{\gamma_{j}}^{\gamma_{j-1}}\, \dd{\nu(x)}=\frac{1}{N}}$ for each $2\le j \le N$,
\beq\begin{split}
	\int_{C_-}^{C_+} \varphi(x)\, \dd{\nu(x)}&=\sum_{i=2}^{N}\int_{\gamma_i}^{\gamma_{i-1}}\varphi(x)\, \dd{\nu(x)}+O(N^{-1})\\& = \sum_{i=2}^{N}\left[\int_{\gamma_i}^{\gamma_{i-1}}[\varphi(x)-\varphi(\gamma_i)]\, \dd{\nu(x)}+\frac{\varphi(\gamma_i)}{N}\right]+O(N^{-1})
\end{split}\eeq
and
\beq
	\begin{split}
		\left|\sum_{i=2}^{N}\int_{\gamma_i}^{\gamma_{i-1}}[\varphi(x)-\varphi(\gamma_i)]\, \dd{\nu(x)}\right |&\le \sum_{i=2}^{N}\int_{\gamma_i}^{\gamma_{i-1}}|\varphi(x)-\varphi(\gamma_i)|\, \dd{\nu(x)} \le \sum_{i=2}^{N}\alpha^2(\gamma_{i-1}-\gamma_i)\int_{\gamma_i}^{\gamma_{i-1}}\, \dd{\nu(x)}\\&\le \frac{\alpha^2(C_+-C_-)}{N} =O(N^{-1}).
	\end{split}
\eeq
Since $\varphi$ is uniformly bounded by $\alpha$, $\varphi(\gamma_1)/N=O(N^{-1})$ and hence,
\beq \label{Integral to Sum}
	\int_{C_-}^{C_+} \varphi(x)\, \dd{\nu(x)}=\frac{1}{N}\sum_{i=1}^N\varphi(\gamma_i)+O(N^{-1}).
\eeq
Combining \eqref{Linear Statistics 1} and \eqref{Integral to Sum},
\beq \label{eq:sum_integral}
	\frac{1}{N}\sum_{i=1}^{N}\varphi(\lambda_i)-\int_{C_-}^{C_+} \varphi(x)\, \dd{\nu(x)}=\frac{1}{N}\left[\sum_{i=1}^{N}\varphi '(\gamma_i)(\lambda_i-\gamma_i)\right]+O(N^{-1}).
\eeq

To estimate the right-hand side of \eqref{eq:sum_integral}, we notice from the uniform boundedness of $\varphi '(\gamma_i)$ and the rigidity estimate that
\beq	
	\frac{1}{N}\left[\sum_{i=1}^{N}\varphi '(\gamma_i)(\lambda_i-\gamma_i)\right]=\frac{\mathcal{Z}}{N}\sum_{i=1}^N\varphi ' (\gamma_{i})\frac{\gamma_{sc,i}}{2}+O_\prec\left(N^{-1}+N^{-(\frac{1}{2}+3\phi)}\right).
\eeq
From Proposition \ref{measure difference}, it is not hard to see that $\gamma_i=\gamma_{sc,i}+O(N^{-2\phi})$. (See also Lemma 2.12 in \cite{he2021fluctuations}.) Thus,
\beq
	\begin{split}
		\frac{\mathcal{Z}}{N}\sum_{i=1}^N\varphi ' (\gamma_{i})\frac{\gamma_{sc,i}}{2}&=\frac{\mathcal{Z}}{N}\sum_{i=1}^N\left[\varphi ' (\gamma_{sc,i})+O(|\gamma_i-\gamma_{sc,i}|)\right]\frac{\gamma_{sc,i}}{2}\\&=\frac{\mathcal{Z}}{N}\left[\sum_{i=1}^{N}\varphi '(\gamma_{sc,i})\frac{\gamma_{sc,i}}{2}\right]+O(\mathcal{Z}N^{-2\phi}),
	\end{split}
\eeq
and we can conclude that
\beq\label{Linear Statistics Approximation 1} 
		\frac{1}{N}\sum_{i=1}^{N}\varphi(\lambda_i)-\int_{C_-}^{C_+} \varphi(x)\, \dd{\nu(x)}=\frac{\mathcal{Z}}{N}\sum_{k=1}^N \left[\varphi'(\gamma_{sc,k})\frac{\gamma_{sc,k}}{2}\right]+O_\prec\left(N^{-1}+N^{-(\frac{1}{2}+3\phi)}\right),
\eeq
where we also used that $\mathcal{Z}\prec N^{-(\phi+\frac{1}{2})}$. Considering the expectations of both sides, we find
\begin{align}
	\label{Linear Statistics Expectation1} \mathbb{E}\left[\frac{1}{N}\sum_{i=1}^{N}\varphi(\lambda_i)\right]&=\int_{C_-}^{C_+} \varphi(x)\, \dd{\nu(x)}+O_\prec\left(N^{-1}+N^{-(\frac{1}{2}+3\phi)}\right),\\ \label{Linear Statistics Expectation2} \frac{1}{N}\sum_{i=1}^{N}\varphi(\lambda_i)-\mathbb{E}\left[\frac{1}{N}\sum_{i=1}^{N}\varphi(\lambda_i)\right]&=\frac{\mathcal{Z}}{N}\sum_{k=1}^N \left[\varphi ' (\gamma_{sc,k})\frac{\gamma_{sc,k}}{2}\right].
\end{align}

For further calculation, we directly solve the equation for $\gamma_{sc,k}$,
\[
	\frac{1}{2\pi}\int_{\gamma_{sc,k}}^{2}\sqrt{4-x^2}\, \dd x=\frac{k}{N}
\] 
as follows. Define 
\[
	\Delta \gamma_{sc,k}:=\gamma_{sc,k}-\gamma_{sc,k-1}, \quad t_k:=\arccos \left(\frac{\gamma_{sc,k}}{2}\right), \quad \Delta t_k:=t_k-t_{k-1},
\]
where we let $\gamma_{sc,0}=2, t_0=0$. From the definition,
\[
	t_k-\cos {t_k}\sin {t_k}=\frac{\pi k}{N}.
\]
We then find 
\[
	(1-\cos {2t_k})\Delta t_k=2\sin^2{t_k}\Delta t_k=\frac{\pi}{N}, \quad \Delta \gamma_{sc,k}=-2\sin t_k\Delta t_k,
\]
which imply the following identity for $\Delta \gamma_{sc, k}$ and $\gamma_{sc, k}$:
\[
	\frac{1}{N\Delta \gamma_{sc,k}}=-\frac{\sin t_k}{\pi}=-\sqrt{1-\left(\frac{\gamma_{sc,k}}{2}\right)^2}.
\]
Using the identity above, we can approximate \eqref{Linear Statistics Approximation 1} by
\beq\label{LSS sum to integral}
	\begin{split}
		\frac{1}{N}\sum_{k=1}^N \varphi ' (\gamma_{sc,k})\frac{\gamma_{sc,k}}{2} &= \sum_{k=1}^N \varphi ' (\gamma_{sc,k})\left(\frac{\gamma_{sc,k}}{2}\right)\left(\frac{1}{N\Delta \gamma_{sc,k}}\right)\Delta \gamma_{sc,k} \\
		&=-\frac{1}{\pi}\sum_{k=1}^N \varphi ' (\gamma_{sc,k})\left(\frac{\gamma_{sc,k}}{2}\right)\sqrt{1-\left(\frac{\gamma_{sc,k}}{2}\right)^2}\Delta \gamma_{sc,k} \\
		&= \frac{1}{\pi}\int_{-2}^2 \varphi ' (x)\left(\frac{x}{2}\right)\sqrt{1-\left(\frac{x}{2}\right)^2}\, \ \dd x+O(N^{-1}) \\
		&=-\frac{1}{2\pi}\int_{-2}^2\varphi (x)\frac{2-x^2}{\sqrt{4-x^2}}\, \dd x+O(N^{-1}),
	\end{split}
\eeq
where we used the integration by parts to obtain the last equality.

Finally, recall that
\beq
	\frac{1}{\sqrt{2}\Sigma}\mathcal{Z}\xrightarrow{d}N(0,1), \quad \Sigma\asymp\frac{1}{N^{\phi+\frac{1}{2}}}
\eeq
and we have already assumed that $N^{\phi+\frac{1}{2}}\Sigma$ converges to $\sigma$ as $N\to\infty$.
Thus,  
\beq
	\frac{q}{\sqrt{N}}\sum_{i=1}^N\varphi(\lambda_i)-\mathbb{E}\left[\frac{q}{\sqrt{N}}\sum_{i=1}^N\varphi(\lambda_i)\right]=-\frac{N^{\phi+\frac{1}{2}}\mathcal{Z}}{2\pi}\int_{-2}^2\varphi (x)\frac{2-x^2}{\sqrt{4-x^2}}\, \dd x+O(N^{-\frac{1}{2}+\phi}\mathcal{Z})
\eeq
converges to a centered Gaussian random variable with the variance
\beq
	V[\varphi]=\frac{\sigma^2}{2\pi^2}\left(\int_{-2}^2\varphi(x)\frac{2-x^2}{\sqrt{4-x^2}}\, \dd{x} \right)^2.
\eeq

It remains to estimate the error term made by changing the measure $\nu$ to $\nu_{sc}$ in the integral $\int_{C_-}^{C_+}\varphi (x)\, \dd\nu (x)$. Following lemma shows that the order of error term is $O(N^{-2\phi})$:
\begin{lem}\label{Change Measure}
	For a function $\varphi$ smooth on an open set containing $\mathrm{[-2, 2]}$, 
\beq
	\int_{C_-}^{C_+}\varphi(x)\, \dd{\nu (x)}-\int_{-2}^{2}\varphi(x)\, \dd{\nu_{sc} (x)}=O(N^{-2\phi}).
\eeq
\end{lem}
This concludes the proof of Proposition \ref{cond:Linear Statistics}.
\end{proof}

\begin{proof}[Proof of Lemma \ref{Change Measure}]
Recall that $\varphi$ is smooth on $[C_-,C_+]$ for any sufficiently large $N$. Since $\varphi$ is bounded in $[C_-,-2]$ and $[2,C_+]$, both of which are intervals of length $O(N^{-2\phi})$,
\beq
	\int_{C_-}^{C_+}\varphi(x)\, \dd{\nu (x)}=\int_{-2}^{2}\varphi(x)\, \dd{\nu (x)}+O(N^{-2\phi})
\eeq
By Proposition \ref{measure difference}, we can change the measure into $\dd\nu_{sc}$ without altering the error term $O(N^{-2\phi})$. Further, since $1/\sqrt{4-x^2}$ is an integrable function of $x$,
\beq
		\left|\int_{-2}^{2}\varphi(x)\, \dd{\nu (x)}-\int_{-2}^{2}\varphi(x)\, \dd{\nu_{sc} (x)}\right|\le CN^{-2\phi}\int_{-2}^{2}\frac{|\varphi(x)|}{\sqrt{4-x^2}}\, \dd x=O(N^{-2\phi}).
\eeq
This concludes the proof of the desired lemma.
\end{proof}

\begin{rem}
	It is clear that \eqref{Linear Statistics Expectation1} holds for a function $\varphi$ that depends on $N$, as long as $\varphi$ satisfies the assumptions in Proposition \ref{cond:Linear Statistics}.
\end{rem}


\subsection*{Acknowledgments}
The work was partially supported by National Research Foundation of Korea under grant number NRF-2019R1A5A1028324 and NRF-2023R1A2C1005843.


\appendix

\section{Technical details} \label{sec:detail}

In this appendix, we provide several technical details in Sections \ref{sec:high} and \ref{sec:low}.

\subsection{Details in Proof of Theorem \ref{thm:sup}} \label{subsec:detail1}

We justify the use of Proposition \ref{cond:Linear Statistics} in the proof of Theorem \ref{thm:sup} where $\varphi(x)=\log (\wh\gamma-x)$ possibly depends on $N$. Since $\wh\gamma=2\beta+\frac{1}{2\beta}+O(N^{-2\phi})$ and $2\beta+\frac{1}{2\beta}$ is uniformly bounded away from $\gamma_{sc,k}$ for all $k$,
\beq \label{Additional_High T}
	\varphi'(\gamma_{sc,k})=-\frac{1}{\wh\gamma-\gamma_{sc,k}}=-\frac{1}{(2\beta+\frac{1}{2\beta})-\gamma_{sc,k}+O(N^{-2\phi})}=-\frac{1}{(2\beta+\frac{1}{2\beta})-\gamma_{sc,k}}+O(N^{-2\phi}).
\eeq
Thus, we find from \eqref{Linear Statistics Expectation2} that
\beq \label{eq:app1}
	\frac{1}{N}\sum_{i=1}^{N}\varphi(\lambda_i)-\mathbb{E}\left[\frac{1}{N}\sum_{i=1}^{N}\varphi(\lambda_i)\right]=-\frac{\mathcal{Z}}{N}\sum_{k=1}^N \left[\frac{1}{(2\beta+\frac{1}{2\beta}-\gamma_{sc,k})}\frac{\gamma_{sc,k}}{2}\right]+O(\mathcal{Z}N^{-2\phi})
\eeq
It in particular shows that the left-hand side of \eqref{eq:app1} converges to a centered Gaussian whose variance is given by (the left-hand side of) \eqref{eq:variance_integral}.

To evaluate the integral in \eqref{eq:variance_integral}, we set $a:=2\beta+\frac{1}{2\beta}-2$. By the definition of the Stieltjes transform of the semicircle measure,
\beq
	\int_{-2}^2\frac{\dd\nu_{sc}(x)}{2+a-x}= m_{sc}(2+a) = \frac{2+a-\sqrt{a(4+a)}}{2},
\eeq
hence by integration by parts,
\beq\begin{split}
	\int_{-2}^2\log (2+a-x)\frac{2-x^2}{\sqrt{4-x^2}}\, \dd x&=\pi\int_{-2}^2\frac{x}{2+a-x}\, \dd\nu_{sc}(x) \\&=\pi (2+a)\int_{-2}^2\frac{\dd\nu_{sc}(x)}{2+a-x}-\pi =4\pi\beta^2.
\end{split}\eeq

\subsection{Evaluation of $F(\beta)$} \label{subsec:detail2}

We want to evaluate $F(\beta)$ defined in \eqref{High T Fbeta} for the high-temperature regime and in \eqref{Low T Fbeta} for the low-temperature regime. We use the notation of $\widehat{\gamma}_{sc}=2\beta+\frac{1}{2\beta}$, which is the counterpart of $\widehat{\gamma}$ when $\nu=\nu_{sc}$. For $F(\beta)$ in the high-temperature regime, we note that
\beq
	\int_{C_-}^{C_+} \log (\widehat\gamma - x) \dd \nu (x) - \int_{C_-}^{C_+} \log (\widehat\gamma_{sc} - x) \dd \nu (x)=\int_{C_-}^{C_+} \log \left(1+\frac{\widehat{\gamma}-\widehat{\gamma}_{sc}}{\widehat{\gamma}_{sc}-x}\right) \dd \nu (x).
\eeq
There exists an $N$-independent constant $c$ such that $\widehat{\gamma}_{sc}-C_+ > c > 0$. Hence, by \eqref{eq:wh_gamma_approx},  
\beq
	\frac{\widehat{\gamma}-\widehat{\gamma}_{sc}}{\widehat{\gamma}_{sc}-x}=O(N^{-2\phi})
\eeq
uniformly for $x\in[C_-,C_+]$, and
\beq
	\int_{C_-}^{C_+} \log (\widehat\gamma - x) \dd \nu (x) - \int_{C_-}^{C_+} \log (\widehat\gamma_{sc} - x) \dd \nu (x)=O(N^{-2\phi}).
\eeq
Applying Lemma \ref{Change Measure} for $\varphi (x)=\log (\widehat{\gamma}_{sc}-x)$, we thus have
\beq \label{Integral estimate1}
		\int_{C_-}^{C_+} \log (\widehat\gamma - x) \dd \nu (x)-\int_{-2}^{2} \log (\widehat\gamma_{sc} - x) \dd \nu_{sc}(x)=O(N^{-2\phi}).
\eeq
For $\beta < 1/2$, it was shown in Appendix in \cite{baik2016fluctuations} that
\beq \label{L sub010}
	\beta \widehat\gamma_{sc}(\beta) - \frac{1}{2} \left( \int_{-2}^{2} \log (\widehat\gamma_{sc}(\beta) - x) \dd \nu_{sc}(x) + 1 + \log (2\beta) \right) = \beta^2,
\eeq
and we can readily check that $F(\beta) = F_0(\beta) + O(N^{-2\phi})$.

Now, we evaluate $F(\beta)$ in the low-temperature regime. We consider the decomposition
\[
	\int_{C_-}^{C_+} \log (C_+-x)\, \dd\nu (x)=\int_{C_-}^{-2} \log (C_+-x)\, \dd\nu (x)+\int_{-2}^{2} \log (C_+-x)\, \dd\nu (x)+\int_{2}^{C_+} \log (C_+-x)\, \dd\nu (x).
\]
By Proposition \ref{cond:regular}, $C_+=-C_-=2+O(N^{-2\phi})$ and $\nu$ exhibits square-root decay. Hence,
\beq
	\int_{2}^{C_+} \log (C_+-x)\, \dd\nu (x)=O(N^{-3\phi}\log N),\quad \int_{C_-}^{-2} \log (C_+-x)\, \dd\nu (x)=O(N^{-2\phi}).
\eeq
Further, applying Proposition \ref{measure difference},
\beq
	\left|\int_{-2}^{2} \log (C_+-x)\, \dd\nu (x)-\int_{-2}^{2} \log (C_+-x)\, \dd\nu_{sc} (x)\right|\le CN^{-2\phi}\int_{-2}^2\frac{|\log (C_+-x)|}{\sqrt{4-x^2}}\, \dd x = O(N^{-2\phi}).
\eeq
It is a simple integration that
\beq\label{Integration}
	\int_{-2}^2\log (2-x+a)\, \dd\nu_{sc} (x) = \frac{1}{2}+a+\frac{a^2}{4}-\frac{(a+2)\sqrt{a(a+4)}}{4}+\log\left(1+\frac{a+\sqrt{a(a+4)}}{2}\right)
\eeq
for $a>0$. By the Taylor expansion, the right hand side of \eqref{Integration} for small $a>0$ is $\frac{1}{2}+a+O(a\sqrt{a})$. Thus, $C_+=2+O(N^{-2\phi})$ implies that
\beq
	\int_{-2}^{2} \log (C_+-x)\, \dd\nu_{sc} (x)-\int_{-2}^{2} \log (2-x)\, \dd\nu_{sc} (x) = O(N^{-2\phi}).
\eeq
Hence,
\beq\label{Integral estimate}
	\int_{C_-}^{C_+} \log (C_+-x)\, \dd\nu (x)-\int_{-2}^{2} \log (2-x)\, \dd\nu_{sc} (x)=O(N^{-2\phi}).
\eeq
For $\beta > 1/2$, it was shown in Appendix in \cite{baik2016fluctuations} that
\beq \label{eq:L sup}
	2\beta  - \frac{1}{2} \left( \int_{-2}^{2} \log (2 - x) \dd \nu_{sc}(x) + 1 + \log (2\beta) \right) = 2\beta-\frac{\log (2\beta)+\frac{3}{2}}{2}
\eeq
from which we can immediately find that $F(\beta) = F_0(\beta) + O(N^{-2\phi})$.

\subsection{Details in Proof of Lemma \ref{Low T Lemma 2}} \label{subsec:detail3}

Recall the definition of $C_+, \gamma_i, \gamma_{sc,i}$. We want to estimate
\[
	\frac{1}{N}\sum_{i=1}^N \frac{\gamma_{sc,i}/2}{C_+-\gamma_i}.
\]
We first decompose the sum into 
\beq
	\frac{1}{N}\sum_{i=1}^N \frac{\gamma_{sc,i}/2}{C_+-\gamma_i}=\frac{1}{N}\sum_{i=1}^{N^{1-2\phi}} \frac{\gamma_{sc,i}/2}{C_+-\gamma_i}+\frac{1}{N}\sum_{i=N^{1-2\phi}+1}^N \frac{\gamma_{sc,i}/2}{C_+-\gamma_i}.
\eeq
Using \eqref{eq:classicalfromC}, the first term in the right hand side can be estimated as
\beq
	\left|\frac{1}{N}\sum_{i=1}^{N^{1-2\phi}} \frac{\gamma_{sc,i}/2}{C_+-\gamma_i}\right|\le\frac{C}{N}\sum_{i=1}^{N^{1-2\phi}}i^{-2/3}N^{2/3}=O(N^{-2\phi/3}).
\eeq
We estimate the second term by using $C_+=2+O(N^{-2\phi})$ and $\gamma_i=\gamma_{sc,i}+O(N^{-2\phi})$. We have
\beq
	\begin{split}\label{Corollary-1}
		\frac{1}{N}\sum_{i=N^{1-2\phi}+1}^N \frac{\gamma_{sc,i}/2}{C_+-\gamma_i}&=\frac{1}{N}\sum_{i=N^{1-2\phi}+1}^N \frac{\gamma_{sc,i}/2}{2-\gamma_{sc,i}+O(N^{-2\phi})}\\&=\frac{1}{2N}\sum_{i=N^{1-2\phi}+1}^N\frac{\gamma_{sc,i}/2}{1-(\gamma_{sc,i}/2)}\left(\frac{1}{1+\left(\frac{O(N^{-2\phi})}{1-(\gamma_{sc,i}/2)}\right)}\right).
	\end{split}
\eeq
We need to estimate $(1-\frac{\gamma_{sc,i}}{2})$ for $N^{1-2\phi}<i\le N$. Clearly, it increases as $i$ increases. By definition, $\gamma_{sc,i}$ satisfies
\beq\label{Corollary-2}
	\arccos(\frac{\gamma_{sc,i}}{2})-\frac{\gamma_{sc,i}}{2}\sqrt{1-\left(\frac{\gamma_{sc,i}}{2}\right)^2}=\frac{\pi}{N}i.
\eeq
Set 
\[
	x := \arccos \left( \frac{\gamma_{sc,N^{1-2\phi}}}{2} \right).
\]
Then, Equation \eqref{Corollary-2} becomes 
\beq\label{Corollary-3}
	x-\cos x\sin x=\pi N^{-2\phi},
\eeq
and by Taylor expanding it,
\beq
	\frac{2}{3}x^3+O(x^5)=\pi N^{-2\phi}.
\eeq
Then, 
\beq
	1-\frac{\gamma_{sc,N^{1-2\phi}}}{2}=1-\cos x=\frac{x^2}{2}+O(x^4)=\sqrt[3]{\frac{9\pi^2}{32}}N^{-4\phi/3}+o(1).
\eeq
Thus, $1-(\gamma_{sc,i}/2)\ge CN^{-4\phi/3}$ for $N^{1-2\phi}< i \le N$ and \eqref{Corollary-1} can be written as
\beq
	\frac{1}{2N}\sum_{i=N^{1-\phi}+1}^N\frac{\gamma_{sc,i}/2}{1-(\gamma_{sc,i}/2)}\left(1+O(N^{-2\phi/3})\right).
\eeq
Similarly,
\begin{align}
	\left|\frac{1}{2N}\sum_{i=1}^{N^{1-2\phi}}\frac{\gamma_{sc,i}/2}{1-(\gamma_{sc,i}/2)}\right|\le\frac{C}{2N}\sum_{i=1}^{N^{1-2\phi}}i^{-2/3}N^{2/3}=O(N^{-2\phi/3}).
\end{align}
Thus, we can conclude that
\beq
	\frac{1}{N}\sum_{i=1}^N \frac{\gamma_{sc,i}/2}{C_+-\gamma_i}=\frac{1}{2N}\sum_{i=1}^N \frac{\gamma_{sc,i}/2}{1-(\gamma_{sc,i}/2)}+O(N^{-2\phi/3}). 
\eeq
Recall that we have shown in \eqref{LSS sum to integral} that 
\beq
	\frac{1}{N}\sum_{k=1}^N \varphi'(\gamma_{sc,k})\frac{\gamma_{sc,k}}{2} =\frac{1}{\pi}\int_{-2}^2 \varphi ' (x)\left(\frac{x}{2}\right)\sqrt{1-\left(\frac{x}{2}\right)^2}\, \ \dd x+O(N^{-1}).
\eeq
Thus, 
\beq
	\frac{1}{2N}\sum_{k=1}^N\frac{\gamma_{sc,k}/2}{1-(\gamma_{sc,k}/2)}=\frac{1}{2\pi}\int_{-2}^2\ \frac{x}{2-x}\sqrt{1-\left(\frac{x}{2}\right)^2}\, \ \dd x+O(N^{-1})=\frac{1}{2}+O(N^{-1})
\eeq
and we can conclude that
\[
	\frac{1}{N}\sum_{i=1}^N \frac{\gamma_{sc,i}/2}{C_+-\gamma_i}=\frac{1}{2}+O(N^{-2\phi/3}).
\]

\section{Proof of Proposition \ref{measure difference}}\label{Proof Measure Difference}
In this appendix, we prove Proposition \ref{measure difference}.
\begin{proof}[Proof of Proposition \ref{measure difference}]
The Proposition \ref{cond:regular} says that $m(z)$ satisfies $1+zm+m^2+N\Sigma^2m^4=0$ with $N\Sigma^2=\frac{\sigma^2}{q^2}+o(N^{-2\phi})$. Define a new function $t(z)$ as
\[
	t(z):=m(z)-m_{sc}(z).
\]
Then,
\beq\label{eq:t-1}
	1+z(m_{sc}+t)+(m_{sc}+t)^2+N\Sigma^2(m_{sc}+t)^4=0,
\eeq
which is
\beq\label{eq:t-2}
	N\Sigma^2m_{sc}^4+(z+2m_{sc})t+t^2=-N\Sigma^2t(4m_{sc}^3+6m_{sc}^2t+4m_{sc}t^2+t^3).
\eeq
We define a new function $f(z)$ as 
\beq
	f(z):=t(z)(4m_{sc}^3+6m_{sc}^2t+4m_{sc}t^2+t^3),
\eeq
which makes the right hand side of \eqref{eq:t-2} as $-N\Sigma^2f(z)$.
By the universality, $t(z)=o(1)$ for each $z$ as $N$ goes to large and the formula of $m_{sc}(z)=(-z+\sqrt{z^2-4})/2$ implies that $|m_{sc}(z)|=O(1)$ uniformly regardless of $z$, and $f(z)=o(1)$. Now, fixing $z$ and a large $N$, we can obtain the recursion formula of $t(z)$ from \eqref{eq:t-2} as
\beq\label{eq:t-3}
	t(z)=\frac{1}{2}\left(-\sqrt{z^2-4}+\sqrt{z^2-4-4N\Sigma^2(m_{sc}^4+f(z))}\right).
\eeq
There are two cases of $z$ based on whether or not $|z^2-4|\ge 4N\Sigma^2|m_{sc}^4+f(z)|$. Recall the simple fact of $|\sqrt{z}|=\sqrt{|z|}$ for a complex number $z$.
\begin{itemize}
	\item If $|z^2-4|\ge 4N\Sigma^2|m_{sc}^4+f(z)|$, then we can do the Taylor expansion of \eqref{eq:t-3}. Hence,
	\beq\label{Taylor expansion of t}
		\begin{split}
			&\left|\sqrt{z^2-4-4N\Sigma^2(m_{sc}^4+f(z))}-\sqrt{z^2-4}\right|\\&=\left|-\frac{2N\Sigma^2(m_{sc}^4+f(z))}{\sqrt{z^2-4}}+\frac{(4N\Sigma^2(m_{sc}^4+f(z)))^2}{(z^2-4)^{3/2}}\sum_{k=2}^\infty\left[\prod_{j=1}^k\left(\frac{j-1-\frac{1}{2}}{j}\right)\right]\left(\frac{4N\Sigma^2(m_{sc}^4+f(z))}{z^2-4}\right)^{k-2}\right|\\&\le \frac{2N\Sigma^2|m_{sc}^4+f(z)|}{\sqrt{|z^2-4|}}+\frac{(4N\Sigma^2|m_{sc}^4+f(z)|)^2}{|z^2-4|^{3/2}}\sum_{k=2}^\infty\left[\prod_{j=1}^k\left|\frac{j-1-\frac{1}{2}}{j}\right|\right]\left|\frac{4N\Sigma^2(m_{sc}^4+f(z))}{z^2-4}\right|^{k-2}\\&\le\sqrt{N\Sigma^2|m_{sc}^4+f(z)|}+2\sqrt{N\Sigma^2|m_{sc}^4+f(z)|}\sum_{k=2}^\infty\left[\prod_{j=1}^k\left|\frac{j-1-\frac{1}{2}}{j}\right|\right].
		\end{split}
	\eeq
When we define a positive sequence $a_n=\prod_{j=1}^n\left|\frac{j-1-\frac{1}{2}}{j}\right|$ for $n\ge 2$, then
\[
	\frac{a_n}{a_{n+1}}=\frac{n+1}{n-\frac{1}{2}},\ \mathrm{satisfying}\ \lim_{n\to\infty}\frac{a_n}{a_{n+1}}=1\ \mathrm{and}\ \lim_{n\to\infty}n\left(\frac{a_n}{a_{n+1}}-1\right)=\frac{3}{2}>1.
\]
Hence, Raabe's test shows that the series is convergent and
\beq\label{bound of |t|-1}
	|t(z)|\le C\sqrt{N\Sigma^2|m_{sc}^4+f(z)|}=O(N^{-\phi})
\eeq
by $N\Sigma^2=\frac{\sigma^2}{q^2}+o(N^{-2\phi})$, and $|m_{sc}(z)|=O(1)$, $f(z)=o(1)$.

	\item If $|z^2-4|\le 4N\Sigma^2|m_{sc}^4+f(z)|$, then 
		\beq\label{bound of |t|-2}\begin{split}
			|t(z)|&\le \frac{1}{2}\left(\sqrt{|z^2-4|}+\sqrt{\left|z^2-4-4N\Sigma^2(m_{sc}^4+f(z))\right|}\right)\\&\le \frac{1}{2}\left(\sqrt{|z^2-4|}+\sqrt{|z^2-4|+4N\Sigma^2\left|m_{sc}^4+f(z)\right|}\right)\\&\le C\sqrt{N\Sigma^2|m_{sc}^4+f(z)|}=O(N^{-\phi}).
			\end{split}
		\eeq
\end{itemize}
The first part of the Proposition \ref{measure difference} is shown by combining \eqref{bound of |t|-1} and \eqref{bound of |t|-2}.

Now, we prove the second part of the proposition. Define two real-valued functions $a(z),\ b(z)$ as the real, imaginary part of $t(z)$, respectively. Then, \eqref{eq:t-2} becomes
\beq\begin{split}\label{eq:a and b}
		&N\Sigma^2m_{sc}^4+\sqrt{z^2-4}(a(z)+\ii b(z))+(a(z)+\ii b(z))^2\\&=-N\Sigma^2(a(z)+\ii b(z))\left(4m_{sc}^3+6m_{sc}^2(a(z)+\ii b(z))+4m_{sc}(a(z)+\ii b(z))^2+(a(z)+\ii b(z))^3\right).
\end{split}\eeq
Now, we try to restrict the domain into $-2\le x\le 2$, setting $x=\mathrm{Re}(z)$, $y=\mathrm{Im}(z)$ and taking the limit of $y\searrow 0$ on the both side of \eqref{eq:a and b}. To do this, we confine the domain of all functions into a compact region containing $\{x+\ii y: -2-\delta<x<2+\delta,\ -\delta<y<\delta \}$ for some $\delta>0$. Then, \eqref{bound of |t|-1} and \eqref{bound of |t|-2} says that $|a(z)|,\ |b(z)|\le |a(z)+\ii b(z)|\le C\sqrt{N\Sigma^2|m_{sc}^4+f(z)|}$ and the domain is restricted to the compact set so there is a constant $C'$ independent on neither $z$ nor $N$ such that $|a(z)|, |b(z)|\le C'/q$. Denote
\beq\label{eq:A and B}
	A(x):=\lim_{y\searrow 0} a(z), \qquad B(x):=\lim_{y\searrow 0} b(z).
\eeq
Then \eqref{eq:A and B} implies that 
\beq\label{estimate A, B}
	|A(x)|,\ |B(x)|\le C'/q
\eeq
for $-2\le x\le 2$. 
Using the fact that $\lim_{y\searrow 0}\sqrt{z^2-4}=\ii \sqrt{4-x^2}$ for $-2\le x\le 2$ (pointwisely,) the real part of \eqref{eq:a and b} after taking the limit becomes
\beq\label{eq:real part}
	\frac{N\Sigma^2}{2}(x^4-4x^2+2)-B\sqrt{4-x^2}+A^2-B^2=-N\Sigma^2\mathrm{Re}\left[\lim_{y\searrow 0}t(z)\left(4m_{sc}^3+6m_{sc}^2t(z)+4m_{sc}t(z)^2+t(z)^3\right)\right].
\eeq
We consider the compact domain so there is a constant $C''$ independent on neither $z$ nor $N$ such that
\beq\label{the other terms}
	|t(z)\left(4m_{sc}^3+6m_{sc}^2t(z)+4m_{sc}t(z)^2+t(z)^3\right)|\le C''N^{-\phi}
\eeq
uniformly in the domain. \eqref{estimate A, B}, \eqref{eq:real part}, and \eqref{the other terms} show that for $-2\le x\le 2$,
\beq\label{estimate of B}
	|B(x)|\sqrt{4-x^2}=\left|A^2-B^2+\frac{N\Sigma^2}{2}(x^4-4x^2+2)\right|+C''(N\Sigma^2)N^{-\phi}\le\frac{C}{q^2}.
\eeq
The inverse Stieltjes transform $\dd \nu (x)-\dd \nu_{sc}(x)=\frac{1}{\pi}B(x) \dd x$ and \eqref{estimate of B} proves the Proposition \ref{measure difference}.
\end{proof}

\begin{rem}
	We do not use \eqref{estimate A, B} in the proof of the main theorems but combining \eqref{estimate A, B} with \eqref{estimate of B} can show more refined version of Proposition \ref{measure difference}: there are constants $C_1,\ C_2$ such that
\[
	|\dd \nu(x)-\dd \nu_{sc}(x)|\le\min\left\{C_1N^{-\phi},\ C_2\frac{N^{-2\phi}}{\sqrt{4-x^2}}\right\}|\dd x|
\]
for $-2\le x\le 2$. We can heuristically check that two terms have same order around $2-|x| = O(N^{-2\phi})$.
\end{rem}


\end{document}